\documentclass[preprint,10pt]{article}
\usepackage{amsmath,amssymb}
\usepackage{mathrsfs}
\usepackage{subfig}
\usepackage{graphicx}
\usepackage{algorithm}
\usepackage{algpseudocode}
\usepackage{color}
\usepackage{textcomp}
\usepackage[table]{xcolor} 
\usepackage{tabularx}
\usepackage{upgreek}
\usepackage{fullpage}
\usepackage{url}
\usepackage{ulem}

\newcommand{\nl}{\nonumber\\}

\begin{document}

\newtheorem{theorem}{Theorem}[section]
\newtheorem{lemma}[theorem]{Lemma}
\newtheorem{proposition}[theorem]{Proposition}
\newtheorem{corollary}[theorem]{Corollary}

\newenvironment{proof}[1][Proof]{\begin{trivlist}
\item[\hskip \labelsep {\bfseries #1}]}{\end{trivlist}}
\newenvironment{definition}[1][Definition]{\begin{trivlist}
\item[\hskip \labelsep {\bfseries #1}]}{\end{trivlist}}
\newenvironment{example}[1][Example]{\begin{trivlist}
\item[\hskip \labelsep {\bfseries #1}]}{\end{trivlist}}
\newenvironment{remark}[1][Remark]{\begin{trivlist}
\item[\hskip \labelsep {\bfseries #1}]}{\end{trivlist}}

\newcommand{\qed}{$\square$}
\newcommand{\comment}[1]{\textcolor{red}{#1}}
\newcommand{\changed}[1]{\textcolor{blue}{#1}}
\newcommand{\wt}[1]{\widetilde{#1}}
\newcommand{\jump}[1]{\text{\textlbrackdbl}#1\text{\textrbrackdbl}}
\makeatletter
\newcommand*{\rom}[1]{\expandafter\@slowromancap\romannumeral #1@}
\newcommand{\dmn}{{\Omega}}
\newcommand{\solid}{{\cal S}}
\newcommand{\restrict}{{\cal R}}
\newcommand{\extend}{{\cal E}}

\newcommand{\cil}{\Omega}
\newcommand{\cilt}{\Omega}

\newcommand{\Qint}[1]{{\cal I}^{#1}_{H}}
\newcommand{\Qintp}{\widetilde{{\cal I}}_{H}}
\newcommand{\norm}[1]{\left\lVert#1\right\rVert}
\newcommand{\trace}[1]{\text{tr}_{0}\left( #1\right)}
\newcommand\abs[1]{\left|#1\right|}
\newcommand{\eps}{\varepsilon}
\newcommand{\de}{\delta}

\newcommand{\dap}{{d}^{2\beta}_{\Lambda}}
\newcommand{\dam}{{d}^{-2\beta}_{\Lambda}}
\newcommand{\dapa}{{d}^{2\alpha}_{\Lambda}}
\newcommand{\dama}{{d}^{-2\alpha}_{\Lambda}}
\newcommand{\Hdir}[1]{{H}^1_{0,\beta}(#1)}
\newcommand{\Hdirdual}[1]{{H}^1_{0,-\beta}(#1)}
\newcommand{\Hspace}[1]{H^1_{\beta}(#1)}
\newcommand{\Lspace}[1]{L^2_{\beta}(#1)}
\newcommand{\Hspacem}[1]{H^1_{-\beta}(#1)}
\newcommand{\Lspacem}[1]{L^2_{-\beta}(#1)}
\newcommand{\Hdirm}[1]{{H}^1_{0,-\beta}(#1)}
\newcommand{\Hdira}[1]{{H}^1_{0,\alpha}(#1)}
\newcommand{\Hdirduala}[1]{{H}^1_{0,-\alpha}(#1)}
\newcommand{\Hspacea}[1]{H^1_{\alpha}(#1)}
\newcommand{\Lspacea}[1]{L^2_{\alpha}(#1)}
\newcommand{\Hspacema}[1]{H^1_{-\alpha}(#1)}
\newcommand{\Lspacema}[1]{L^2_{-\alpha}(#1)}
\newcommand{\Hdirma}[1]{{H}^1_{0,-\alpha}(#1)}
\newcommand{\TwoNorm}[2]{{\left\lVert#1\right\rVert}_{L^2_{\beta}\left(#2\right)}}
\newcommand{\TwoNormminus}[2]{{\left\lVert#1\right\rVert}_{L^2_{-\beta}\left(#2\right)}}
\newcommand{\HNorm}[2]{{\left\lVert#1\right\rVert}_{H^1_{\beta}\left(#2 \right)}}
\newcommand{\TwoNormm}[2]{{\left\lVert#1\right\rVert}_{L^2_{-\beta}\left(#2\right)}}
\newcommand{\HNormm}[2]{{\left\lVert#1\right\rVert}_{H^1_{-\beta}\left(#2 \right)}}

\newcommand{\TwoNorma}[2]{{\left\lVert#1\right\rVert}_{L^2_{\alpha}\left(#2\right)}}
\newcommand{\TwoNormminusa}[2]{{\left\lVert#1\right\rVert}_{L^2_{-\alpha}\left(#2\right)}}
\newcommand{\HNorma}[2]{{\left\lVert#1\right\rVert}_{H^1_{\alpha}\left(#2 \right)}}
\newcommand{\TwoNormma}[2]{{\left\lVert#1\right\rVert}_{L^2_{-\alpha}\left(#2\right)}}
\newcommand{\HNormma}[2]{{\left\lVert#1\right\rVert}_{H^1_{-\alpha}\left(#2 \right)}}

\newcommand{\local}{C_{{\cal I}_{H}}}
\newcommand{\localone}{C_{{\cal I}_{H,1}}}
\newcommand{\localtwo}{C_{{\cal I}_{H,2}}}
\newcommand{\localthree}{C_{{\cal I}_{H,3}}}
\newcommand{\localprime}{{C'}_{{\cal I}_{H}}}

\newcommand{\avrg}[2]{{\langle #1 \rangle}_{#2}}

\newcommand{\hK}{\tilde{K}}
\newcommand{\x}{{\bf x}}
\newcommand{\y}{{\bf y}}
\newcommand{\z}{{\bf z}}
\newcommand{\vnode}{{\bf v}}
\newcommand{\vnodep}{{\bf v}'}
\newcommand{\vnodepp}{{\bf v}''}
\newcommand{\wnode}{{\bf w}}
\newcommand{\wnodep}{{\bf w}'}
\newcommand{\wnodepp}{{\bf w}''}
\newcommand{\tr}[2]{\text{tr}_{#1}\left(   #2    \right)}

\makeatother

\title{Upscaling  Singular Sources in Weighted Sobolev Spaces by Sub-Grid Corrections}

\author{%
	Donald L. Brown \thanks{University of Nottingham, School of 
		Mathematical Sciences,  \mbox{donald.brown@nottingham.ac.uk}}%
	\and %
	Joscha Gedicke \thanks{Faculty of Mathematics, University of Vienna, 1090 Vienna, Austria, 
		\mbox{joscha.gedicke@univie.ac.at}}%
}

\maketitle

\begin{abstract}
In this paper, we develop a numerical multiscale method to solve elliptic boundary value problems with heterogeneous diffusion coefficients and with singular source terms. 
When the diffusion coefficient is heterogeneous, this adds to the computational costs, and this difficulty is compounded by a singular source term. For singular source terms, the solution
does not belong to the Sobolev space $H^1$, but to the space $W^{1,p}$ for some $p<2$. Hence, the problem may be reformulated in a distance-weighted Sobolev space.  
Using this formulation, we develop a method to upscale the multiscale coefficient  near the singular sources by incorporating corrections into the coarse-grid. 
Using a sub-grid correction method, we correct the basis functions in a distance-weighted Sobolev space and show that these corrections can be truncated 
to design a computationally efficient scheme with optimal convergence rates. 
Due to the nature of the formulation in weighted spaces, the variational form must be posed on the cross product of complementary spaces. 
Thus, two such sub-grid corrections must be computed, one for each multiscale space of the cross product. 
A key ingredient of this method is the use of quasi-interpolation operators to construct the fine scale spaces.
Therefore, we develop a weighted projective quasi-interpolation that can be used for a general class of Muckenhoupt weight functions. 
We verify the optimal convergence of the method in some numerical examples with singular point sources and line fractures, and with oscillatory and
heterogeneous diffusion coefficients.
\end{abstract}


\section{Introduction }
 
Computing flow in heterogeneous porous media is a difficult problem due to the high-contrast in material properties as well as the large disparate scales of the permeability or hydraulic conductivity. To simplify the calculation, an upscaled or effective model is preferred so that many models and scenarios may be tested. The computational upscaling, or  numerical homogenization, of complex porous media has a large literature in various areas of applications in petroleum, environmental, and materials engineering. One key aspect, particularly in subsurface modeling, is the upscaling of material properties in  the neighborhood of singular sources, i.e. near wells or fractured injection/production sites. The upscaling of numerical simulations near the singular wellbore source in petroleum engineering has its roots in the work of Peaceman \cite{peaceman1983interpretation}. Here, special care must be taken in upscaling near the well as it is modeled by a singular Dirac source at the production site.  There are various procedures for upscaling near wells in subsurface modeling, cf. \cite{chen2009well,durlofsky2005upscaling} for a general survey. In addition to point sources, complex fracture networks of linear or planar type sources are also considered and often need to be upscaled for fast efficient simulation \cite{gilman2003practical,gong2008upscaling,li2015effective}. 
\par
The simulation of the fine-grid (non-upscaled) problem also poses unique challenges in this setting. Given the standard regularity, a $H^{-1}$ source, material properties that are $L^\infty$-elliptic, and a sufficiently regular domain, the solution of  the elliptic partial differential equation is in the Hilbert space $H^1$. However, from classical results of Stampacchia \cite{stampacchia1963equations}, for measure or singular sources such as Dirac measures, the solution lies in a Banach space $W^{1,p}$, with $p<2$ or can be reformulated in a fractional Hilbert space $H^s$, $s\in (0,1)$ \cite{babuvska1971error,scott1973finite}.
Recently, the authors in \cite{agnelli2014posteriori,d2012finite,d2008coupling}, consider reformulating the problem in a class of weighted Sobolev spaces, that are also Hilbert spaces, where square gradients are bounded in weighted norms. 
Then, the authors apply and analyze a finite element method in weighted spaces,
with weights that belong to the general class of Muckenhoupt weights \cite{kufner1985weighted,muckenhoupt1972weighted,nochetto2016piecewise}. These types of weighted spaces have proven to be very useful in the analysis and computation of the fractional Laplacian and its related extension \cite{brown2017numerical,caffarelli2007extension,capella2011regularity,cabre2010positive,nochetto2015pde}.
 In this work, as was analyzed in \cite{d2012finite,d2008coupling}, we consider an elliptic problem with $L^\infty$ coefficients with a singular source term in distance-weighted Sobolev spaces. Using the weight $\dap(x)=\text{dist}(x,\Lambda)^{2\beta}$, where $\Lambda$ is the support of a singular source and for some specific range of $\beta$. In this work, we will suppose that $\Lambda$ is either a point, a line fracture, or a planar-type fracture, depending on the ambient space dimension. The analysis of the continuous and the discrete problem for line-fractures in three dimensional space was carried out in \cite{d2012finite,d2008coupling}, and for Dirac point sources in two and three dimensions in \cite{agnelli2014posteriori}. In this work, we extend the analysis also to the case of fracture-line sources in two dimensions and planar-type fractures in three dimensions using a general result on traces in distance-weighted Sobolev spaces \cite{nekvinda1993characterization}.
\par
Returning to the issue of upscaling singular sources and multiscale features, we will utilize the theory of weighted Sobolev spaces and combine it with a multiscale method based on the variational multiscale method \cite{MR2300286,MR1660141} and its localization theory \cite{Henning.Morgenstern.Peterseim:2014,MP12}. As is the case with upscaling in the subsurface modeling literature, there are a vast array of approaches to numerical homogenization. A few approaches are  the 
multiscale finite element method \cite{Hou:Wu:1997},  
the heterogeneous multiscale method \cite{Abdulle:E:Engquist:Vanden-Eijnden:2012},   and the variational multiscale method \cite{MR1660141}. 
We will employ the  local orthogonal decomposition (LOD) method, a sort of localization of the variational multiscale method. The LOD method
is a numerical upscaling method whereby  the coarse-grid is augmented so that the corrections are 
localizable and truncated to design a computationally efficient scheme 
\cite{HP13,Kornhuber.Peterseim.Yserentant:2016,MP11,Peterseim:2015}. 
This method has been utilized in many applications \cite{Henningwave,brown2016multiscaleelastic,Brown2017,brown2016multiscale,henning2014localized,HMP12},
and has been used successfully in other weighted space contexts such as the fractional Laplacian \cite{brown2017numerical}, and many of the techniques derived in this setting will be used here. We prove the optimal error  
estimates for the LOD method with the ideal multiscale spaces as well as with truncated corrections \cite{brown2017numerical,HMP12}.
\par
A key component of the LOD upscaling method is a quasi-interpolation operator that is utilized to construct a fine-scale space.
The authors in \cite{nochetto2016piecewise} utilize a quasi-interpolation based on regularized Taylor polynomials 
\cite{brenner2007mathematical}, which are a generalization of the Cl\'{e}ment quasi-interpolation 
\cite{clement1975approximation} and are analyzed for a general class of Muckenhoupt weights, of which $\dap$ belongs to for certain intervals of $\beta$.  However, to obtain a projective quasi-interpolation, we proceed similarly
as in \cite{brown2017numerical}, where the authors utilized a local distance-weighted $L^2$ projection onto the coarse-grid 
space. Then, we prove local $L^2$ stability and approximability  properties in weighted Sobolev spaces. This proof and the truncation arguments are left for the appendix.
\par
%
We present numerical results for two different diffusion coefficients, one highly oscillatory and the other with heterogeneous data taken from the SPE10 benchmark.
We show that we obtain numerically optimal convergence rates in case of a point source, a point source together with a point sink, and for a line fracture in two dimensional space,
for a range of admissible values for $\beta$.
In all numerical experiments, we obtain good computational efficiency by truncating the computation of the correctors. 

\par
This paper is organized as follows. We begin in Section \ref{probsetting}, where we introduce the elliptic problem with discontinuous coefficients and with singular source terms. We then sketch the theory of weighted Sobolev spaces for the weights $\dap$. Here, we outline the key ingredients of well-posedness such as Poincar\'{e} and trace inequalities, $L^2$-type decompositions, and a-priori bounds. In Section \ref{quasi_int_section}, we introduce stability and approximability for the quasi-interpolation operator. Then, in Section \ref{MSmethod} we construct a multiscale space to upscale the heterogeneities and singularities of the source term. The method differs here from standard approaches in that two multiscale spaces must be computed for the bilinear form that is defined on a cross-product space. 
We then derive the global and truncated error bounds in Section \ref{errorsection}. 
 {
In Section \ref{numericssection}, we present the results of some numerical examples with two different diffusion coefficients, different singular source terms, and various suitable values for $\beta$. 
}
Finally, the proof for the stability and approximability  of the quasi-interpolation is given in Appendix \ref{stabAppendix} and the proofs for the  truncation of correctors in distance-weighted norms are given in Appendix  
\ref{truncproofsection}.
 
\section{ Problem Setting and Background}\label{probsetting}
In this section we will introduce the problem setting and some notation  for the  relevant distance-weighted Sobolev spaces. We introduce the idea of a Muckenhoupt weight, which yields a class of weighted spaces that have a valid Poincar\'{e} inequality. For a certain subclass of distance-weighted exponents we have a trace inequality from the singular source to the interior of the domain. This fact, along with a useful $L^2-$type decomposition will give us well-posedness, as well as a-priori bounds.
\subsection{Elliptic Problems with Singular Sources}
Let $\Omega\subset \mathbb{R}^d$ be a bounded, open, and connected  domain  for $d=2,3$, with Lipschitz boundary.
 We seek to solve the following heterogeneous Laplace equation with  Dirichlet boundary condition  for $u$
 \begin{subequations}\label{hetLaplace}
 	\begin{align}
 	-\text{div} \left(A\nabla u\right)&=f \delta_{\Lambda} \text{ in } \dmn,\\
 	\label{dirichlet}
 	u&=0 \text{ on } \partial \dmn,
 	\end{align}
 \end{subequations}
 where $\delta_{\Lambda}$ is the Dirac mass on $\Lambda$. Here $\Lambda$ is a sufficiently smooth closed submanifold of $\Omega$, such that dim$(\Lambda)=\ell<d$, for $\ell=0,1,2$. For simplicity we suppose, $\ell=0$ is  a point-source $x_{0}$, $\ell=1$ corresponds to a piecewise line fracture, and for $\ell=2$ this corresponds to a planar-type fracture.
\par
 
We suppose that the coarse-grid size $H$ is constrained so that the singular objects are contained inside a coarse-grid element. In other words,
all of these objects are assumed to be small, i.e. $\operatorname{diam}(\Lambda)\approx{\cal O}(H)$ and $|\Lambda|\approx H^\ell$. Further, we suppose $f\in L^2(\Lambda)$, where if $\Lambda=x_{0}$, then $f$ is just a finite constant.
 Here $A\in \left(L^{\infty}(\dmn)\right)^{d\times d}$ is assumed to be symmetric and satisfy for $x\in \dmn$
 \begin{align*}
 \gamma_1 |\xi|^2\leq \avrg{A\xi,\xi}{}\leq \gamma_2 |\xi|^2,
 \end{align*}
 for some $\gamma_1,\gamma_2>0$ and all $\xi\in\mathbb{R}^d$.

\subsection{Weighted Sobolev Spaces}
To facilitate the solution of \eqref{hetLaplace} we need additional notation.
We define the following class of  weighted Sobolev spaces for a positive weight $\dap$.  For $x\in \mathbb{R}^d$, let $dx$ be the Lebesgue measure on $\mathbb{R}^{d}$, and $ds$ on $\mathbb{R}^{\ell}$. 
We will use the notation $A\lesssim B$ if there exists a $C>0$ such that $A<CB$, where $C$ is independent of the mesh, but may depend on other parameters such as $\beta,d, \ell, \Omega, \gamma_1,\gamma_2$, etc.
For an open set ${\omega}\subset \mathbb{R}^{d}$, we define $\Lspace{\omega}$ to be all measurable functions $u$ on ${\omega}$, such that 
\[
{{\left\lVert u \right\rVert}_{L^2_{\beta}\left( {\omega} \right)}}=\int_{{\omega}} u^2\, \dap \, dx < \infty,
\]
for ${\beta} \in \left(-\frac{d-\ell}{2},\frac{d-\ell}{2}\right)$, so that the weight is of Muckenhoupt class $A_2(\mathbb{R})$, cf. \eqref{Muckenhouptconstant}.
Define $\Hspace{{\omega}}$ similarly,  all measurable functions $u$ on ${\omega}$, such that
\[ \HNorm{u}{{\omega}}:=\left(\TwoNorm{u}{{\omega}}^2+\TwoNorm{\nabla u}{{\omega}}^2\right)^{\frac{1}{2}}< \infty, \]
and we denote the space incorporating the vanishing boundary condition as
\[ \Hdir{{\omega}}=\{u\in\Hspace{{\omega}}\, : \, u=0 \text{ on } \partial{\omega}\}. \]
Integrating \eqref{hetLaplace} by parts we obtain the following weak form. We seek a solution $u\in \Hdir{\Omega}$, so that
\begin{align}\label{varform}
a(u,\psi)=l(\psi)  \text{ for all }  \psi\in \Hdirdual{\Omega},
\end{align}
where $a(\cdot, \cdot):  \Hdir{\Omega}\times  \Hdirm{\Omega}\to \mathbb{R}$ is the bilinear form
\[ a(u, \psi)=\int_{\Omega} A \nabla u \nabla \psi dx,\]
and we suppose, with more generality than the source term in  \eqref{hetLaplace}, that  $l\in ( \Hdirm{\Omega})'$.
\par

A key property of the distance weight $\dap$ is that it belongs to the Muckenhoupt class $A_2(\mathbb{R}^{d})$, \cite{gol2009weighted,haroske2008atomic,muckenhoupt1972weighted,nochetto2016piecewise}.
 For a general weight,
 $w \in L_{loc}^1(\mathbb{R}^{d})$, we say that $w \in A_{p}(\mathbb{R}^{d})$ if there exists a $C_{p,w}>0$ such that 
 \begin{align}\label{Muckenhouptconstant}
 \sup_{B} \left(\frac{1}{|B|}\int_{B} w \,  dx  \right)\left(\frac{1}{|B|}\int_{B} w^{\frac{1}{1-p}}\,   dx\right)^{p-1} =C_{p,w}<\infty,
 \end{align}
 for all balls $B\subset\mathbb{R}^{d}$. We will denote the Muckenhoupt weight constant for $w$ as $C_{p,w}$. 
 It can be shown, for very general sets (even fractal) $\Lambda$, that the following proposition holds.
 \begin{proposition}\label{alphaMuckenhoupt}
 Suppose that 
$ \beta \in \left( -\frac{d-\ell}{2},\frac{d-\ell}{2}\right), $
 then the  weight $d_{\Lambda}^{2 \beta}\in A_2(\mathbb{R}^d)$, 
 where dim$(\Lambda)=\ell$. More explicitly, we have for balls $B$ in $\mathbb{R}^d$ that 
\begin{align}\label{Muckenhouptconstant.distance}
 \sup_{B} \left(\frac{1}{|B|}\int_{B} \dap \,  dx  \right)\left(\frac{1}{|B|}\int_{B}\dam \,   dx\right)^{} =C_{2,\beta}<\infty.
\end{align}
\end{proposition} 
\begin{proof}
The case of a point-source, $ \ell=0$ and $d=2,3$, can be found in \cite{agnelli2014posteriori}. The case of a linear fracture, $\ell=1$ and $d=3$  in \cite{d2008coupling,d2012finite}. The general case can be found in \cite{haroske2008atomic} and references therein. \qed
 \end{proof}
\par

 The key inequality that holds existence, uniqueness, and the general analysis together is the weighted Poincar\'{e} inequality.
 The weighted Poincar\'{e} inequality for Muckenhoupt weights is well studied in nonlinear potential theory of degenerate problems \cite{fabes1982local,kufner1985weighted,heinonen2012nonlinear} and references therein.
 \begin{lemma}[Distance-Weighted Poincar\'{e} Inequality]\label{poincare2}
Let $\omega\subset\Omega$, be a bounded, star-shaped domain (with respect to the ball B) and $\text{diam}(\omega)\approx H$.  Suppose that 
$\beta \in \left( -\frac{d-\ell}{2},\frac{d-\ell}{2}\right)$,
if $w\in\Hspace{\omega}$   then we have
\begin{align}\label{poincare_unit2}
 	\norm{w-\avrg{w}{\omega}}_{L_{\beta}^2(\omega )  }\lesssim H \norm{\nabla w}_{L_{\beta}^2(\omega )  },
\end{align}
where the constants are independent of $H$ and $\avrg{w}{\omega}=\frac{1}{|\omega|}\int_{\omega} w \, dx$.
 \end{lemma}
 \begin{proof}
By Proposition \ref{alphaMuckenhoupt}, we have $\dap\in A_{2}(\mathbb{R}^d)$, and so by the general Muckenhoupt weighted  Poincar\'{e} inequality \cite[Corollary 3.2]{nochetto2016piecewise}, we easily obtain the result. 
 \qed
 \end{proof}
 
 \begin{remark}
 	As noted in \cite{brown2017numerical}, the above inequality may be extended to a connected union of star-shaped domains where the 
 	average can be taken over a subdomain. This can be proven in a similar way to \cite[Corollary 4.4]{nochetto2015pde}.  We will refer to both of these 
 	results simply as the weighted Poincar\'{e} inequality when there is no ambiguity.
 Further, for completeness,
 	we note a similar Friedrich's type inequality also holds for $w\in \Hdir{\omega}$,
	\[ \norm{w }_{L_{\beta}^2(\omega )  }\lesssim H \norm{\nabla w}_{L_{\beta}^2(\omega )  }. \]
 \end{remark}
 \par
 We have the following decomposition of $L^2_{\beta}(\dmn;\mathbb{R}^d)$ that is critical for existence and uniqueness of solutions to \eqref{varform} in weighted spaces.
 \begin{lemma}[Decomposition of $L_{\beta}^2(\Omega;\mathbb{R}^d)$] \label{spacedecomp}
 	Let $\beta \in \left( -\frac{d-\ell}{2},\frac{d-\ell}{2}\right)$, for $\boldsymbol{\tau}\in L^2_{\beta}(\dmn;\mathbb{R}^d)$ there exist a pair $(\boldsymbol{\sigma}, z)\in L^2_{\beta}(\dmn;\mathbb{R}^d) \times \Hdir{\dmn}$ such that
\begin{align*}
\boldsymbol{\tau}=\nabla z+\boldsymbol{\sigma}, \, \,  \left<A \boldsymbol{\sigma} , \nabla w\right>_{\Omega}=0, \forall w\in \Hdirm{\dmn},\\
\TwoNorm{\nabla z}{\Omega}\lesssim \TwoNorm{\boldsymbol{\tau}}{\Omega}, \, \, \TwoNorm{ \boldsymbol{\sigma} }{\Omega}\lesssim \TwoNorm{\boldsymbol{\tau}}{\Omega}.
\end{align*}
 \end{lemma}
 \begin{proof}
 	 Since we have the weighted Poincar\'{e}-Friedrich's inequalities from Lemma \ref{poincare2} 
	 for $\beta \in \left( -\frac{d-\ell}{2},\frac{d-\ell}{2}\right)$, we see that an immediate generalization of \cite[Lemma 2.1]{d2012finite} is possible, 
	 and the same abstract  proof holds.  \qed
 \end{proof}
 From this Lemma, as in \cite{agnelli2014posteriori,d2012finite} we  establish the well-posedness of the abstract problem \eqref{varform}.
 \begin{theorem}\label{existence}
 	Let
 	$\beta \in \left( -\frac{d-\ell}{2},\frac{d-\ell}{2}\right)$, then the abstract problem \eqref{varform} is well-posed, and we have the following stability bound
 	\begin{align}
 	\norm{u}_{ \Hdir{\Omega}}\lesssim 	\norm{l}_{\left(  \Hdirm{\Omega} \right)'}.
 	\end{align}	
 \end{theorem}
 \begin{proof}
 	This is an immediate corollary of Lemma \ref{spacedecomp}, cf.  \cite[Corollary 2.2]{d2012finite}.    \qed
 \end{proof}
 
The above theorem is for more general source terms than we will consider in this work. We will focus on singular source terms and so must consider a smaller class of function spaces, and values here,
which we will denote as  $\alpha$. To this end, we introduce the natural trace space related to Dirac measures $\delta_{\Lambda}$.

 \begin{lemma}[Distance-Weighted Trace Inequality]\label{CanonicalTrace}
 	Suppose dim$(\Lambda)=\ell$, $\text{dim}(\Omega)=d$, and $0\leq \ell \leq d-1$,  and that $\alpha$ is so that
 	\begin{align}\label{alphabounds}
 	\frac{d-\ell}{2}-1< \alpha <\frac{d-\ell}{2}.
 	\end{align}
 	Then, there exists a bounded continuous trace operator $\tr{\Lambda}{\cdot}: L^2(\Lambda)\to \Hspacema{\Omega}$. We have the following  bound
 	\begin{align}\label{CanonicalTraceestimate}
 	\norm{v}_{L^{2}(\Lambda)}\lesssim \HNormma{  v}{ \dmn },
 	\end{align}
 	where the hidden constant depends on $\alpha$ and $\Lambda$.
 \end{lemma}
 \begin{proof}
 	The case of a point-source, $ \ell=0$ and $d=2,3$, can be found in \cite{agnelli2014posteriori}. The case of a linear fracture, $\ell=1$ and $d=3$  in \cite{d2012finite,d2008coupling}. 
 	For a general discussion on trace spaces of distance-weighted spaces we refer   to \cite{nekvinda1993characterization}, where one can see the general bounds; in particular for the case of planar type fractures $\ell=2$ and $d=3$, as well as the case of a linear fracture $\ell=1$ in $d=2$ dimensional space. \qed
 \end{proof}
 Thus, we have the following  well-posedness for singular source terms.
  \begin{corollary}\label{apriori}
  	Suppose dim$(\Lambda)=\ell$, $\text{dim}(\Omega)=d$, and $0\leq \ell \leq d-1$.	
	Let $\alpha \in \left( \frac{d-\ell}{2}-1,\frac{d-\ell}{2}\right)$, 
	then     \eqref{hetLaplace} with Dirac measure data, $f\delta_{\Lambda}$,  for $f\in L^2(\Lambda)$, is well-posed, and we have the following  stability bound
  	\begin{align}
  	\norm{u}_{ \Hdira{\Omega}}\lesssim 	\norm{f}_{L^2(\Lambda)}.
  	\end{align}
  \end{corollary}
 \begin{proof}
This is an immediate corollary of Lemma \ref{spacedecomp} and the trace Lemma \ref{CanonicalTrace}. The arguments can be see in more detail in \cite[Remark 1]{d2012finite}.
 \qed
 \end{proof}

\section{Quasi-Interpolation in  Distance-Weighted  Sobolev Spaces}\label{quasi_int_section}
The multiscale method utilized in this paper, as well as previous works \cite{brown2016multiscaleelastic,brown2017numerical,brown2016multiscale,MP12}, relies on the construction of a projective quasi-interpolation operator.
Here we construct a quasi-interpolation operator for distance-weighted Sobolev spaces using weighted local $L^2$ projections onto  simplices in a similar vein to the authors in  \cite{bramble2002stability}.  Much of this presentation will follow that of \cite{brown2017numerical}, where the authors handled a specific type of weight for fractional Laplacians. 
We introduce the discretization and classical nodal basis.  
We then state the local stability and approximability properties of these operators.

\subsection{Coarse Grid Finite Elements}
Here we follow much of the notation in \cite{brown2017numerical,MP11}.  We suppose that we have a coarse quasi-uniform, shape-regular discretization ${\cal T}_{H}$ of the  domain $\Omega$ with characteristic mesh size $H$.  In this work, we will not consider errors from the fine-grid $h$. 
We denote the nodes of the mesh  ${\cal N}$. The interior   nodes of $\Omega$ (not including vanishing Dirichlet condition) we denote as ${\cal N}_{int}$, and the Dirichlet nodes as ${\cal N}_{dir}$.  We will write ${\cal N}({\omega})$ for nodes restricted to ${\omega}$, similarly for interior,  or Dirichlet nodes.
We suppose further that there is a $T\in {\cal T}_{H}$ such that $\Lambda \subset T$. Note that if the source intersected a small number of triangles, taking the intersected patch would also be sufficient here. 

Let the classical conforming $\mathbb{P}_{1}$ 
finite element space over ${\cal T}_{H}$ be given by $S_{H}$,
and let $V_{H}=S_{H}\cap \Hdir{\dmn}$. Utilizing the notation in \cite{nochetto2015pde},  we denote $\vnode \in{\cal N}_{}$ as nodal values. 
The $\mathbb{P}_1$ nodal basis function $\lambda_{\vnode}$, for a node $\vnode \in{\cal N}_{}$, is written as
\begin{align}\label{P1}
\lambda_{\vnode}(\vnode)=1 \text{ and } \lambda_{\wnode}(\vnode)=0, \vnode\neq \wnode\in {\cal N}_{} .
\end{align}
This is a basis for  $V_{H}$.
We define the  patch around $\vnode$ as
\[
\omega_{\vnode}=\bigcup_{T \ni \vnode}T,
\]
for $T\in {\cal T}_{H}$.
We define for any patch ${\omega}_{\vnode}$
the extension patch
\begin{subequations}\label{patches}
	\begin{align}
	{\omega}_{\vnode}&={\omega}_{\vnode,0}=\text{supp}( \lambda_{\vnode}) \cap \dmn,\\
	{\omega}_{\vnode,k}&=\text{int}(\cup \{ T\in {\cal T}_{H}| T\cap \bar{\omega}_{\vnode,k-1}\neq \emptyset \}\cap \dmn,
	\end{align}
\end{subequations}
for $k\in \mathbb{N}_{+}$.  
We will denote $V_{H}|_{\omega}$ to be the coarse grid space restricted to some domain ${\omega}$. 
\par
We further suppose that for these patches $\frac{|B|}{|{\omega}_{\vnode,k}|}\lesssim 1$, for some ball $B$ containing ${\omega}_{\vnode,k}$. 
Thus, for $\beta \in \left( -\frac{d-\ell}{2},\frac{d-\ell}{2}\right)$, we have the bound 
\begin{align}
    &\left(\frac{1}{|{\omega}_{\vnode,k}|}\int_{{\omega}_{\vnode,k}} \dap \, dx  \right)\left(\frac{1}{|{\omega}_{\vnode,k}|}\int_{{\omega}_{\vnode,k}} \dam\, dx \right)  \lesssim \left(\frac{|B|}{|{\omega}_{\vnode,k}|  }\frac{1}{|B|} \int_{B} \dap \, dx  \right)\left(\frac{|B|}{|{\omega}_{\vnode,k}|}\frac{1}{|B|} \int_{B} \dam \, dx \right)\nonumber \\
    \label{Muckenhouptconstant.patch}
   & \lesssim \left(\frac{|B|}{|{\omega}_{\vnode,k}|  } \right)^2 \left(\frac{1}{|B|} \int_{B} \dap \, dx  \right)\left(\frac{1}{|B|} \int_{B} \dam \, dx \right) \lesssim \left(\frac{|B|}{|{\omega}_{\vnode,k}|  } \right)^2  C_{2,\beta} \lesssim C_{2,\beta}.
\end{align}
Where we utilized the bound \eqref{Muckenhouptconstant} and Proposition \ref{alphaMuckenhoupt}, hence we can apply the Muckenhoupt weight bounds to our patches.

\subsection{Quasi-Interpolation Operator}
In a related setting, 
the authors in \cite{nochetto2015pde,nochetto2016piecewise} construct a quasi-interpolation based on a higher order Cl\'{e}ment type of operator.
In this section, we will construct a slightly different quasi-interpolation that is also projective. 
This projective quasi-interpolation satisfies the requisite stability and approximability properties. 
This is a  modification of the operator of \cite{bramble2002stability} and was utilized in perforated domains in \cite{brown2016multiscale} and in \cite{brown2017numerical} for fractional Laplacians.
\par
We now define the $\dap$-weighted local $L^2$ projections, for  
$\beta \in \left( -\frac{d-\ell}{2},\frac{d-\ell}{2}\right)$. For $\vnode\in{\cal N}_{int}$,
${\cal P}_{\vnode}: L_{\beta}^2({\omega}_{\vnode})\to {V}_{H}|_{\omega_{\vnode}}$
is a weighted projection in the sense that
\begin{align}\label{L2proj}
\int_{{\omega}_{\vnode}}( {\cal P}_{\vnode}u) v_{H} \dap  \, dx= \int_{{\omega}_{\vnode}}(u v_{H}) \dap  \, dx \text{ for all } v_{H}\in  {V}_{H}|_{\omega_{\vnode}}.
\end{align}
From this we define the quasi-interpolation operator $\Qint{\beta}: H^1_{\beta}(\Omega ) \to {V}_{H}$ as
\begin{align}\label{proj}
\Qint{\beta} u(x)=\sum_{\vnode\in{\cal N}_{int}} ({\cal P}_{\vnode} u)(\vnode) \lambda_{\vnode}(x) .
\end{align}
Note that this quasi-interpolation assumes zero Dirichlet boundary conditions as we sum over the interior nodes ${\cal N}_{int}$ only. 
If we have non-trivial Dirichlet conditions, techniques to handle the boundary nodes will have to be employed as in \cite{brown2017numerical,scott1990finite}, 
or even additional boundary corrections have to be computed as in \cite{henning2014localized}.

\subsection{  Local Stability and Approximability  }
The quasi-interpolation operator $\Qint{\beta}$ defined by \eqref{proj} satisfies the following stability and local approximation properties.  
The proof of this lemma is based on that presented in  \cite{brown2017numerical,melenkApel}, 
but is slightly simpler here since we do not have to treat non-zero boundary terms. Since the proof is only slightly different, 
we leave it for the Appendix \ref{stabAppendix}.
\begin{lemma}\label{stablelemma}
Let $\Qint{\beta}$ be given by \eqref{proj} and $\vnode\in {\cal N}$. 
Suppose that $\beta \in \left( -\frac{d-\ell}{2},\frac{d-\ell}{2}\right)$.
The quasi-interpolation satisfies the following local stability estimates for all $u\in H^{1}_{\beta}(\Omega)$,
\begin{subequations}\label{stableoperator}
	\begin{align}
		\TwoNorm{\Qint{\beta} u}{{\omega}_{\vnode}  }&\lesssim \TwoNorm{ u}{{\omega}_{\vnode,1}  }, \\ 
		\TwoNorm{\nabla \Qint{\beta} u}{{\omega}_{\vnode}  }&\lesssim \TwoNorm{\nabla u}{{\omega}_{\vnode,1}  }.
		\end{align}
	\end{subequations}
	Further, the quasi-interpolation satisfies the following local approximation properties
	\begin{subequations}\label{stableproj}
		\begin{align}
		\label{stableproj1}
		&\TwoNorm{u-\Qint{\beta} u}{{\omega}_{\vnode}  } \lesssim   H \TwoNorm{\nabla  u}{{\omega}_{\vnode,1} },\\
		\label{stableproj2}
		&\TwoNorm{\nabla (u-\Qint{\beta} u)}{ {\omega}_{\vnode} } \lesssim   \TwoNorm{\nabla  u}{\omega_{\vnode,1} }.
		\end{align}
	\end{subequations}
	Moreover, the quasi-interpolation $\Qint{\beta}$ is a projection.
\end{lemma}
\begin{proof}
	See Appendix \ref{stabAppendix}. \qed
\end{proof}

\section{Numerical Upscaling Method }\label{MSmethod}
We now will construct our multiscale approximation space to handle the oscillations created by the heterogeneities of the coefficient and the sub-grid singular source terms.
The singular source terms are incorporated into the coarse-grid corrections. 
This splitting can be found in \cite{Henning.Morgenstern.Peterseim:2014,MP11} and references therein.
We begin by constructing fine-scale spaces, that contain the small scale information, as well as singular source information via the distance weight.

\subsection{Construction of the Multiscale Space}
We define the kernel quasi-interpolation operator for 
$\beta \in \left( -\frac{d-\ell}{2},\frac{d-\ell}{2}\right)$
to be the subspace
\begin{align*}
{V}_{\beta}^f=\{v\in  \Hdir{\dmn}  \;| \; \Qint{\beta} v=0\}.
\end{align*}
These  spaces will capture the sub-grid scale singular features not resolved by ${V}_H$. Note that by stability of $\Qint{\beta}$, this is  a closed subspace of $\Hdir{\Omega}$, as this will be needed for the fine-scale decomposition of $L^2_{\beta}(\Omega)$. 
We define the corrector $Q^{\beta}_{\dmn}  :{V}_H\to {V}_{\beta}^f$ to be the projection operator such that for $v_H\in {V}_H$ we compute $Q^{\beta}_{\dmn} (v_{H})\in {V}^f_{\beta}$ as
\begin{align}\label{corrector}
\int_{\dmn}A\nabla Q^{\beta}_{\dmn} (v_H) \nabla w\,  { \dap } \,   \, dx=\int_{\dmn}A\nabla v_H \nabla w\,  { \dap} \,   \, dx, \text{ for all } w\in  { {V}_{\beta}^f}.
\end{align}
We use the correctors to define the multiscale space
\begin{align}
\label{msspace_def}
{V}^{ms,\beta}_{H}=({V}_{H}-Q^{\beta}_{\dmn} ( {V}_{H}))\subset \Hdir{\dmn}.
\end{align}
This projection gives a $A\dap-$weighted orthogonal splitting
\[
\Hdir{\dmn}={V}^{ms,\beta}_{H}\oplus_{A\dap} {V}^f_{\beta},
\]
so that for  $u \in \Hdir{\dmn}$ and $u_{H}^{ms}\in{V}^{ms,\beta}_{H}$, we have $u-u_{H}^{ms}\in {V}^f_{\beta}$.  Further, we note in the following Proposition \ref{correctorPropBound} that these correctors are well posed, and thus the multiscale space exists.

\begin{proposition}\label{correctorPropBound}
	The corrector problem \eqref{corrector} is well posed and $Q^{\beta}_{\dmn}$ satisfies the bound
	\begin{align}
	\TwoNorm{\nabla Q^{\beta}_{\dmn}(v_{H})}{\Omega}\lesssim \TwoNorm{\nabla v_H}{\Omega}.
	\end{align}
\end{proposition}
\begin{proof}
	It is trivial to see that the variational form \eqref{corrector} is coercive and bounded on the closed subspace $V_{\beta}^f\subset \Hdir{\dmn}$, so the Lax-Milgram theorem holds. The difficulty is to obtain the bounds and to make sure that the right-hand side is well posed. 
	In problem \eqref{corrector}, take $w\in{V}_{\beta}^f$ to be $Q^{\beta}_{\dmn}(v_H)$, then we have 
	\begin{align*}
	\TwoNorm{\nabla Q^{\beta}_{\dmn}(v_{H})}{\Omega}^2 &\lesssim 	\int_{\dmn}|A^{1/2}\nabla Q^{\beta}_{\dmn} (v_H)|^2  \,  { \dap } \,   \, dx=\int_{\dmn}A\nabla v_H\nabla Q^{\beta}_{\dmn}(v_{H}) \,  { \dap} \,   \, dx\\
	&\lesssim \left(\int_{\dmn}|A^{1/2}\nabla v_H|^2 \,  { \dap} \,   \, dx \right)^{\frac{1}{2}}\left(\int_{\dmn}|A^{1/2}\nabla Q^{\beta}_{\dmn}(v_{H})|^2 \,  { \dap} \,   \, dx \right)^{\frac{1}{2}} .
	\end{align*}
	Here we used the coercivity and the boundedness of the bilinear form, and we obtain 
	\[
	\TwoNorm{\nabla Q^{\beta}_{\dmn}(v_{H})}{\Omega} \lesssim  \left(\int_{\dmn}|\nabla v_H|^2 \,  { \dap} \,   \, dx \right)^{\frac{1}{2}}\lesssim \TwoNorm{\nabla v_H}{\Omega}< \infty .
	\]
	This is due to the fact that  $V_{H}\subset \Hdir{\Omega}$ for $\beta  \in \left( -\frac{d-\ell}{2},\frac{d-\ell}{2}\right)$,
	which  can be seen from two cases.  The trivial case is when $\beta\in (0,\frac{d-\ell}{2})$, and so $\norm{\dap}_{L^{\infty}} \lesssim C(\Omega)<\infty$, and thus
	\[
	\left(\int_{\dmn}|\nabla v_H|^2 \,  { \dap} \,   \, dx \right)^{\frac{1}{2}}\lesssim  C(\Omega)\left(\int_{\dmn}|\nabla v_H|^2 \,  \,   \, dx \right)^{\frac{1}{2}}< \infty.
	\]
	Now we suppose $\beta\in (-\frac{d-\ell}{2},0)$, and denote with $T$ the triangle where $\nabla v_H$ obtains its maximum, then 
	 {
	\begin{align*}
	&    \left(\int_{\dmn}|\nabla v_H|^2 \,  { \dap} \,   \, dx \right)^{\frac{1}{2}} \lesssim  \norm{\nabla v_H}_{L^{\infty}(T)}  \left(\int_{\dmn} \,  { \dap} \,   \, dx \right)^{\frac{1}{2}} < \infty
	\end{align*}
	}
	%
	since  $ { \dap} \in L^1(\Omega)$, if $\beta\in (-\frac{d-\ell}{2},0)$, which can be seen from \cite[Lemma 2.2]{agnelli2014posteriori}.  \qed
\end{proof}

Note here this $\beta$ interval is the most general, and only takes into account the values where the distance function is of Muckenhoupt class. For a singular source term, we need the restricted interval $\alpha\in  \left( \frac{d-\ell}{2}-1,\frac{d-\ell}{2}\right)$. 
 The multiscale problem is  defined  on a cross product:
	\[
	V^{ms,  \alpha}_{H}\times V^{ms,  -\alpha}_{H}\subset \Hdira{\dmn} \times \Hdirduala{\dmn}.
	\]
We refer to these  modified coarse spaces, $V_{H}^{ms \pm \alpha}$, as the ``ideal'' multiscale spaces. The multiscale Galerkin approximation  $u_{H}^{ms}\in {V}_{H}^{ms,\alpha}$ to \eqref{hetLaplace}  satisfies 
\begin{align}\label{msvarform}
\int_{\dmn}A \nabla u_{H}^{ms} \nabla v\,    dx=\int_{\Lambda } f v \, ds \text{ for all } v\in {V}^{ms,-\alpha}_{H} .
\end{align}
\begin{remark}
	 Note that we must have this $\pm\alpha$ pairing due to   the bilinear form acting on the cross product
		\[
		a(\cdot, \cdot): V^{ms,  \alpha}_{H}\times V^{ms,  -\alpha}_{H}\subset \Hdira{\dmn} \times \Hdirduala{\dmn} \to \mathbb{R}.
		\]
		In addition, due to the requirements of the trace theorem for singular data, we must have 
		\[ v\in {V}^{ms,-\alpha}_{H}\subset\Hdirduala{\dmn} \]
		 so that error bounds may be obtained.
\end{remark}

\subsection{Truncated Multiscale Space}
The solution of  \eqref{corrector} requires the calculation of  global correctors.  However, is is now well established that in most diffusive regimes the correctors decay exponentially. 
To this end, we define the localized fine-scale space to be the fine-scale space extended by zero outside the patch, that is in the larger $\beta$ interval
\[
{V}^f_{\beta}({\omega}_{\vnode,k})=\{v\in {V}_{\beta}^f |\text{   } v|_{\dmn\backslash {\omega}_{\vnode ,k}}=0\}.
\]
We let  for some $\vnode \in {\cal N}_{int}$ and $k\in \mathbb{N}$ the localized  corrector operator $Q^{\beta}_{\vnode,k}: {V}_{H}\to {V}_{\beta}^f({\omega}_{\vnode,k})$, be defined such that given a $v_{H}\in {V}_{H}$
\begin{align}\label{Qcorrector}
\int_{{\omega}_{\vnode,k}} A \nabla Q^{\beta}_{\vnode ,k}(v_{H}) \nabla w\,   {\dap} \, dx=\int_{{\omega}_{\vnode}}A\hat{\lambda}_{\vnode}\nabla v_{H} \nabla w\,   \,  {\dap}  dx, \text{ for all } w \in {V}^f_{\beta}({\omega}_{\vnode ,k}),
\end{align}
where $\hat{\lambda}_{\vnode}=\frac{{\lambda}_{\vnode}}{\sum_{\vnode'\in {\cal N}_{in }}\lambda_{\vnode'}}$ is  augmented due to the zero Dirichlet condition. 
The collection $\{ \hat{\lambda}_{\vnode }\}_{\vnode\in {\cal N}_{in}}$ is a partition of unity \cite{Henning.Morgenstern.Peterseim:2014}.  
We denote the global truncated corrector operator as
\begin{align}\label{Qcorrectorbasis.global}
Q^{\beta}_{k}(v_{H})=\sum_{\vnode\in {\cal N}_{int}}Q^{\beta}_{\vnode,k}(v_{H}).
\end{align}
With this  notation, we   write the truncated multiscale space as
\begin{align}
\label{msspace_def_trunc}
{V}^{ms,\beta}_{H,k}=({V}_{H}-Q^{\beta}_{k} ( {V}_{H}))\subset \Hdir{\dmn}.
\end{align}
Then, the corresponding truncated multiscale approximation to \eqref{hetLaplace} is: find $ u_{H,k}^{ms}\in {V}^{ms,\alpha}_{H,k}$ such that
\begin{align}\label{localmsvarform}
\int_{\dmn}A \nabla u_{H,k}^{ms} \nabla v\,   \, dx=\int_{\Lambda}f v ds \text{ for all } v\in {V}^{ms,-\alpha}_{H,k}.
\end{align}
This more efficient scheme is utilized in the numerical experiments. 
\par
Note also that for sufficiently large $k$, we recover the full domain and obtain the ideal corrector with functions of global support, denoted $Q^{\beta}_{\dmn }$, from \eqref{corrector}.

\section{Error Analysis}\label{errorsection}
In this section we present the error introduced by using  \eqref{msvarform} on the global domain to compute the solution to \eqref{hetLaplace}. Then, we show how localization effects the error when we use  \eqref{localmsvarform} on truncated domains. The key component of these error estimates is related to the trace spaces from Lemma  \ref{CanonicalTrace} and the a-priori estimate from Corollary \ref{apriori}.

\subsection{Weighted Trace Inequality}
We begin first by a scaled weighted trace inequality. 
\begin{lemma}\label{traceHlemma} Suppose that 
	$\alpha \in \left(\frac{d-\ell}{2}-1,\frac{d-\ell}{2}\right)$.
	Let $T\in {\cal T}_{H}$ such that  $\Lambda \subset   T$.  Then, for $u\in H^1_{-\alpha}(\Omega)$, we have the following trace inequality 
	\begin{align}\label{traceboundH}
	\norm{ u}_{L^2 (\Lambda) }   \lesssim H^{\alpha-\frac{d-\ell}{2}   } \TwoNormminusa{  u}{T } +H^{\alpha-\frac{d-\ell}{2}+1   }   \TwoNormminusa{  \nabla u}{ T }.
	\end{align}
\end{lemma}
\begin{proof}
	We proceed by using mapping arguments similar to \cite[Lemma 7.2]{ErnGuermond2015} and weighted-scaling arguments from \cite{d2012finite,d2008coupling} for $\ell=1$ and $d=3$. The estimate for $\ell=0$ and $d=2,3$ can be found in \cite{agnelli2014posteriori}. 
	Using general scaling arguments, we generalize this to the case of $\ell=2$ and $d=3$, and $\ell=1$ and $d=2$.
	\par
	We denote the  
	the reference (unit size) element $\hat{T}$ and similarly the reference sub-domain $\hat{\Lambda}$. We let $A_{T}:\hat{T}\to T$ be an affine mapping, and denote $\hat{u}=u\circ A_{T}$, $\hat{x}=A_T^{-1}(x)$ for $x\in T$. Clearly, $\hat{\Lambda}=A^{-1}_T(\Lambda)\subset \hat{T}$ and $\text{diam}(T)\approx H$.
	Note that from \cite[Lemma 3.2]{d2012finite} we have from {shape regularity} that  {$ c H d_{\hat{\Lambda}}(\hat{x}) \leq d_{\Lambda}(A_{T}(\hat{x}))\leq C H d_{\hat{\Lambda}}(\hat{x})$}, thus, 
	\begin{align}\label{scalingK}
	\TwoNormma{u}{T}^2=\int_{T} u^2 \text{d}_{\Lambda}^{-2\alpha}(x) dx\gtrsim\frac{|T|}{|\hat{T}|}H^{-2 \alpha}\int_{\hat{T}} \hat{u}^2(\hat{x}) d_{\hat{\Lambda}}^{-2\alpha}(\hat{x}) d\hat{x} \gtrsim C H^{-2\alpha} \frac{|T|}{|\hat{T}|} \norm{\hat{u}}^2_{L^2_{-\alpha}(\hat{T})}.
	\end{align}
	By using standard trace inequality arguments,  the trace bound \eqref{CanonicalTraceestimate}, and the above scaling \eqref{scalingK}, in the weighted norm we obtain 
	\begin{align*}
	\norm{{u}}_{L^2(\Lambda)}&=\left( \frac{|\Lambda|}{|\hat{\Lambda}|} \right)^{\frac{1}{2} }  \norm{{\hat{u}}}_{L^2(\hat{\Lambda})}\lesssim   
	|\Lambda| ^{\frac{1}{2} }  \left(  \TwoNormma{  \hat{u}}{ \hat{T} }+\TwoNormma{  \nabla \hat{u}}{ \hat{T} }\right)\\
	&  \lesssim |\Lambda| ^{\frac{1}{2} } |T|^{-\frac{1}{2}}H^{\alpha} \left(  \TwoNormma{  {u}}{ {T} }+\norm{\nabla A_{T}}\TwoNormma{  \nabla {u}}{ {T} }\right)\\
	&  \lesssim  H^{\frac{\ell}{2}} H^{-\frac{d}{2}}  H^{\alpha} \left(  \TwoNormma{  {u}}{ {T} }+H\TwoNormma{  \nabla {u}}{ {T} }\right).
	\end{align*}
	Here we have used that  $|\Lambda| \approx H^\ell$, for $\ell=0,1,2$, where we take  $|\Lambda|=|x_0|=1$ for $\ell=0$, and
	where $|\cdot|$ refers to the measure in the relevant dimension for $\ell=1,2$. Here we suppose a planar fracture has area $H^2$ and a line fracture has length $H$.   \qed
\end{proof}

Using local approximability of $\Qint{\alpha}(u)$,  we have the following corollary.
\begin{corollary}\label{traceHlemmaInteropolant} Suppose the assumptions of Lemma \ref{traceHlemma}, we then have
	\begin{align}\label{traceboundH}
	\norm{ u-\Qint{-\alpha}(u)}_{L^2 (\Lambda) }   \lesssim H^{\alpha-\frac{d-\ell}{2}+1   }   \TwoNormminusa{  \nabla u}{ T }.
	\end{align}
\end{corollary}
\begin{proof}
This is an easy consequence of Lemma \ref{traceHlemma}, and stability and approximability of $\Qint{-\alpha}(u)$  from Lemma \ref{stablelemma}.
\end{proof}

\subsection{Error with Global Support}
To obtain the error of the multiscale method with globally computed correctors \eqref{corrector}, we must utilize the tools of existence and uniqueness for the cross-product space as in \cite{d2012finite}. To this end, we have the following fine-scale decomposition of $L^2_{\beta}(\Omega;\mathbb{R}^d)$. This will then allow us to prove an error bound for the upscaling method.
\begin{theorem}[Fine-Scale Decomposition of $L_{\beta}^2(\Omega;\mathbb{R}^d)$] \label{fine-scale-decomp}Let $\beta   \in \left(-\frac{d-\ell}{2},\frac{d-\ell}{2}\right)$. For each $\boldsymbol{\tau} \in L^2_{\beta}(\dmn;\mathbb{R}^d)$, there is a unique pair $(\boldsymbol{\sigma},z) \in L^2_{\beta}(\dmn;\mathbb{R}^d)\times V_{\beta}^f$ so that 
	\begin{align*}
&\boldsymbol{\tau} =\nabla z+\boldsymbol{\sigma}, \, \, \int_{\dmn}A\boldsymbol{\sigma}\nabla w \, dx =0, \,\, \forall w\in V_{-\beta}^f,\\
&\TwoNorm{\nabla z}{\dmn}\lesssim \TwoNorm{ \boldsymbol{\tau} }{\dmn}, \, \, \TwoNorm{\boldsymbol{\sigma}}{\dmn}\lesssim \TwoNorm{ \boldsymbol{\tau} }{\dmn}.
	\end{align*}
This can be written as the direct sum: $L_{\beta}^2(\dmn;\mathbb{R}^d)=\nabla V_{\beta}^f \oplus \left(\nabla V_{-\beta}^f\right)^{\perp_{A}}$
\end{theorem}
\begin{proof}
	Here the proof is the same as \cite[Lemma 2.1]{d2012finite}, with $M_1=L_{\beta}^2(\dmn;\mathbb{R}^d)$, $M_2=L_{-\beta}^2(\dmn;\mathbb{R}^d)$, $X_1=V_{\beta}^f$, and $X_2=V_{-\beta}^f$. This is primarily  due to having appropriate Poincar\'{e} inequalities in this setting. 
 \qed
\end{proof}
%
We have the following error for the approximation computed from  \eqref{msvarform}.
\begin{theorem}\label{errorglobal}
	Let 
	$\alpha \in \left(\frac{d-\ell}{2}-1,\frac{d-\ell}{2}\right)$.
	Suppose that $u\in \Hdira{\dmn}$ satisfies \eqref{varform} with source term $f\delta_{\Lambda}$, $f\in L^2(\Lambda)$, and that $u_{H}^{ms}\in {V}_{H}^{ms,\alpha}$ satisfies \eqref{msvarform}.  Then, we have the following error estimate
	\begin{align}\label{eq:errorglobal}
	\TwoNorma{\nabla u-\nabla u_{H}^{ms}}{\dmn}\lesssim H^{\alpha-\frac{d-\ell}{2}+1   } \norm{f}_{L^2(\Lambda)}.
	\end{align}
\end{theorem}
\begin{proof}
 {
	From the orthogonal splitting of the spaces we have that  $ u- u_{H}^{ms}=u^f \in {V}^f_{\alpha}$. 
	Thus,  by taking $\boldsymbol{\tau} = \nabla u^f \dapa\in L_{-\alpha}^2(\dmn;\mathbb{R}^d)$ and by Theorem \ref{fine-scale-decomp}, there exists a $\boldsymbol{\sigma} \in L_{-\alpha}^2(\dmn;\mathbb{R}^d)$ and  a $v^f \in V^f_{-\alpha}$ so that $\boldsymbol{\tau}=\nabla v^f+\boldsymbol{\sigma}$, where 
	  $ \int_{\Omega}A \nabla u^f \boldsymbol{\sigma}\,  dx =0$,
	  and $ \TwoNorma{ \nabla u^f}{\Omega}= \TwoNormminusa{\boldsymbol{\tau} }{\Omega}\gtrsim  \TwoNormma{\nabla  v^f}{\Omega}$.
	  We have
	\begin{align*}
  	\int_{\Omega}A \nabla u^f \nabla v^f dx =  \int_{\Omega}A \nabla u^f \boldsymbol{\tau} dx-\int_{\Omega}A \nabla u^f \boldsymbol{\sigma} dx \gtrsim \TwoNorma{\nabla u^f}{\Omega}^2\gtrsim \TwoNorma{\nabla  u^f}{\Omega}\TwoNormma{\nabla  v^f}{\Omega}.
	\end{align*}		    
	Using the above together with $\Qint{-\alpha}(v^f)=0$, and the trace estimate of Corollary \ref{traceHlemmaInteropolant},
	we have 	
	\begin{align*}
	\TwoNorma{\nabla  u^f}{\Omega}\TwoNormma{\nabla  v^f}{\Omega}\lesssim & \int_{\dmn}A \nabla u^f \nabla v^f\,   dx=\int_{\Lambda }  f  (v^f-\Qint{-\alpha}(v^f)) ds\\
	&\lesssim \norm{f}_{L^2(\Lambda)}\norm{ v^f-\Qint{-\alpha}(v^f)}_{L^2(\Lambda)} \lesssim   H^{\alpha-\frac{d-\ell}{2}+1   }   \TwoNormminusa{  \nabla v^f}{ \dmn} \norm{f}_{L^2(\Lambda)}.
	\end{align*}
 Dividing the last $\TwoNormma{\nabla v^f}{\dmn}$ term yields the result.   \qed
 }
\end{proof}

\subsection{Error with Localization}
In this  section,  we show the error due to truncation with respect to patch extensions. The standard result holds here, similarly to that in \cite{brown2017numerical} and the references therein.
The key lemma needed is the following estimate, the proof is   standard and for completeness  can be found in Appendix \ref{truncproofsection}. 
\begin{lemma}\label{localglobal.derp}
	Let $\beta \in \left(-\frac{d-\ell}{2},\frac{d-\ell}{2}\right)$ and  $u_{H}\in {V}_{H}\subset \Hdir{\Omega}$, let $Q^{\beta}_{k}$ be constructed from \eqref{Qcorrector} and \eqref{Qcorrectorbasis.global}, and $Q^{\beta}_{\dmn}$ defined to be the ``ideal" corrector without truncation in \eqref{corrector}, then
	\begin{align}
	&	\TwoNorm{\nabla( Q^{\beta}_{\dmn}(u_{H})-Q^{\beta}_{k}(u_{H}))  }{\dmn} 	\lesssim  k^{\frac{d}{2}} \theta^{k  }   \TwoNorm{ \nabla u_{H} }{\dmn}, 
	\end{align}
	with  $\theta\in (0,1)$.
\end{lemma}
\begin{proof}
	See Appendix \ref{truncproofsection}. \qed
\end{proof}
We then are able to derive an error bound with localized correctors.
\begin{theorem}\label{errorlocal}
	Let 
	$\alpha \in \left(\frac{d-\ell}{2}-1,\frac{d-\ell}{2}\right)$.
	Suppose that $u\in \Hdira{\dmn}$ satisfies \eqref{varform} with source term $f\delta_{\Lambda}, f\in L^2(\Lambda)$, and that $u_{H,k}^{ms}\in {V}_{H,k}^{ms,\alpha}$, with local correctors  calculated from \eqref{Qcorrector},
	satisfies \eqref{localmsvarform}.  Then, we have the following error estimate 
	\begin{align}\label{eq:errorlocal1}
	\TwoNorma{\nabla u-\nabla u_{H,k}^{ms}}{\dmn}\lesssim \left(  H^{\alpha-\frac{d-\ell}{2}+1   } + k^{\frac{d}{2}}  \theta^{k  }\right)\norm{f}_{L^2(\Lambda)},
	\end{align}
	with $\theta \in (0,1)$.
\end{theorem}
 
\begin{proof}
	We follow the proof given in \cite{brown2017numerical}. We let $u^{ms}_{H}=u_{H}-Q^\alpha_{\cilt}(u_{H})$ be the ideal global multiscale solution satisfying \eqref{msvarform},
	and $u^{ms}_{H,k}=u_{H,k}-Q^\alpha_{k}(u_{H,k})$ be the corresponding truncated solution to \eqref{localmsvarform}. Then, by Galerkin approximations being energy minimizers we have
	\[
	\TwoNorma{\nabla u-\nabla (u_{H,k}-Q^\alpha_{k}(u_{H,k}))}{\dmn}\lesssim \TwoNorma{\nabla u-\nabla( u_{H}-Q^\alpha_{k}(u_{H}) )}{\dmn}.
	\]
	Using this fact and
	Theorem \ref{errorglobal} and Lemma
	\ref{localglobal.derp} we have
	\begin{align*}
	\TwoNorma{\nabla u-\nabla u_{H,k}^{ms}}{\dmn}&\lesssim\TwoNorma{\nabla u-\nabla( u_{H}-Q^\alpha_{\dmn}(u_{H})+Q^\alpha_{\dmn}(u_{H})-Q^\alpha_{k}(u_H))}{\dmn} \\
	&\leq\TwoNorma{\nabla u-\nabla u_{H}^{ms}}{\dmn}+\TwoNorma{\nabla( Q^\alpha_{\dmn}(u_{H})-Q^\alpha_{k}(u_{H}))  }{\dmn}\\
	&\lesssim H^{\alpha-\frac{d-\ell}{2}+1  } \norm{f}_{L^2(\Lambda)}+ k^{\frac{d}{2}}  \theta^{k   } \TwoNorma{ \nabla u_{H}  }{\dmn}.
	\end{align*}
	In addition note that, by construction $ u_{H}^{}=\Qint{\alpha}(u_{H}^{ms})$.
	Thus, using  local stability \eqref{stableproj2} and a-priori bounds from \eqref{msvarform}, obtained via the trace inequality in Lemma \ref{CanonicalTrace}  and  Corollary \ref{apriori}, we have
	\begin{align*}
	\TwoNorma{ \nabla u^{}_{H} }{\dmn} =
	\TwoNorma{ \nabla \Qint{\alpha}(u_{H}^{ms}) }{\dmn}\lesssim \TwoNorma{ \nabla u_{H}^{ms} }{\dmn}\lesssim \norm{f}_{L^2(\Lambda)}.
	\end{align*}
 	Thus, applying the above, we obtain our bound. 
	\qed
\end{proof}

\section{Numerical Examples}\label{numericssection}
In this section we present numerical experiments for the unit square
$\Omega = (0,1)^2$ with three different singular source terms. First,
a single singular source with $f=1$ at $(1/2,1/2)$, then a singular sink and source with $f=-1$ at $(3/4,1/4)$ and $f=1$ at $(1/4,3/4)$, finally a singular line fracture with $f=1$ along $(3/2^3,1/2)\times(5/2^3,1/2)$.
As indicated from the theory, for the point singular sources we consider $\alpha\in (0,1)$, while for the
line fracture we consider the range $\alpha\in (-1/2,1/2)$.
In the following we present the results for two different types of
multiscale permeabilities $A$, the first one being  a highly oscillatory periodic coefficient and the second one 
is constructed from the SPE10 benchmark data.
For numerical efficiency we chose $k=3$ layers for the localized corrector problems
in all numerical experiments.
Since for each problem below the solution $u$ is unknown, we compare the multiscale approximations 
$u_{H,3}^{ms}$, for $H=2^{-1},\ldots,2^{-5}$, to a reference approximation $u_h$ on
a fine grid with $h=2^{-9}$.

\subsection{Highly Oscillatory Example}
\begin{figure}
\centering
\includegraphics[width=0.32\textwidth]{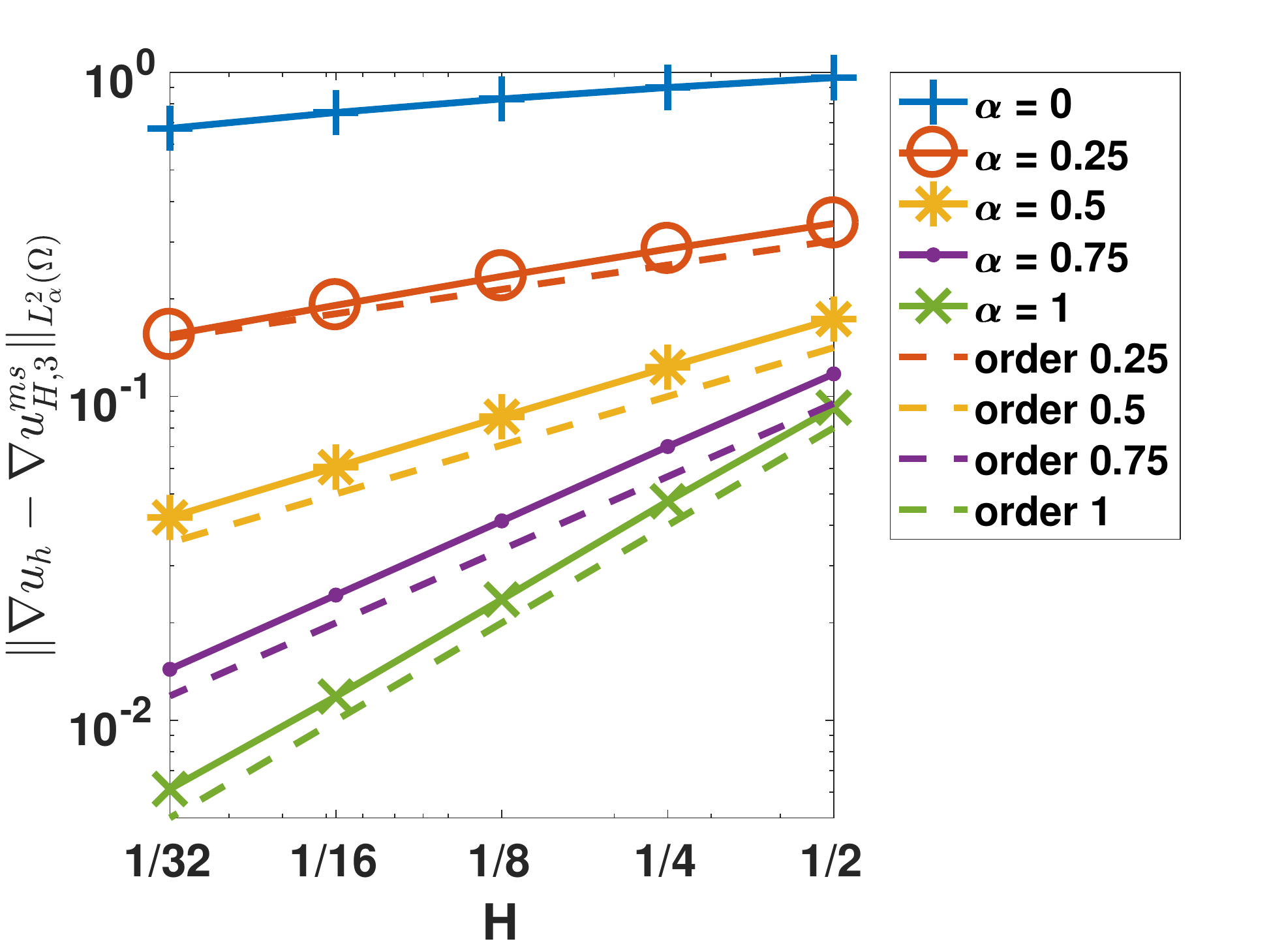}
\includegraphics[width=0.32\textwidth]{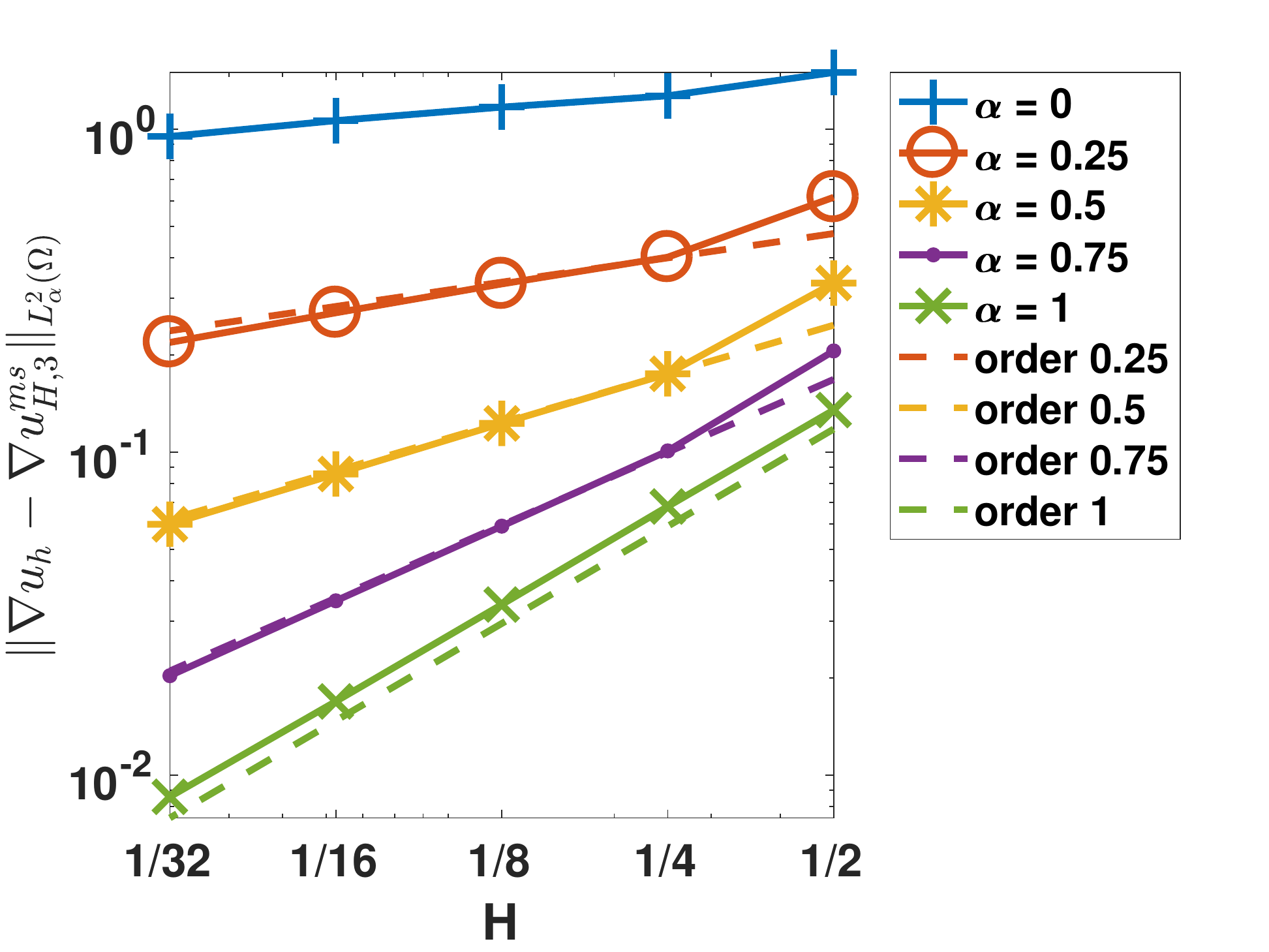}
\includegraphics[width=0.32\textwidth]{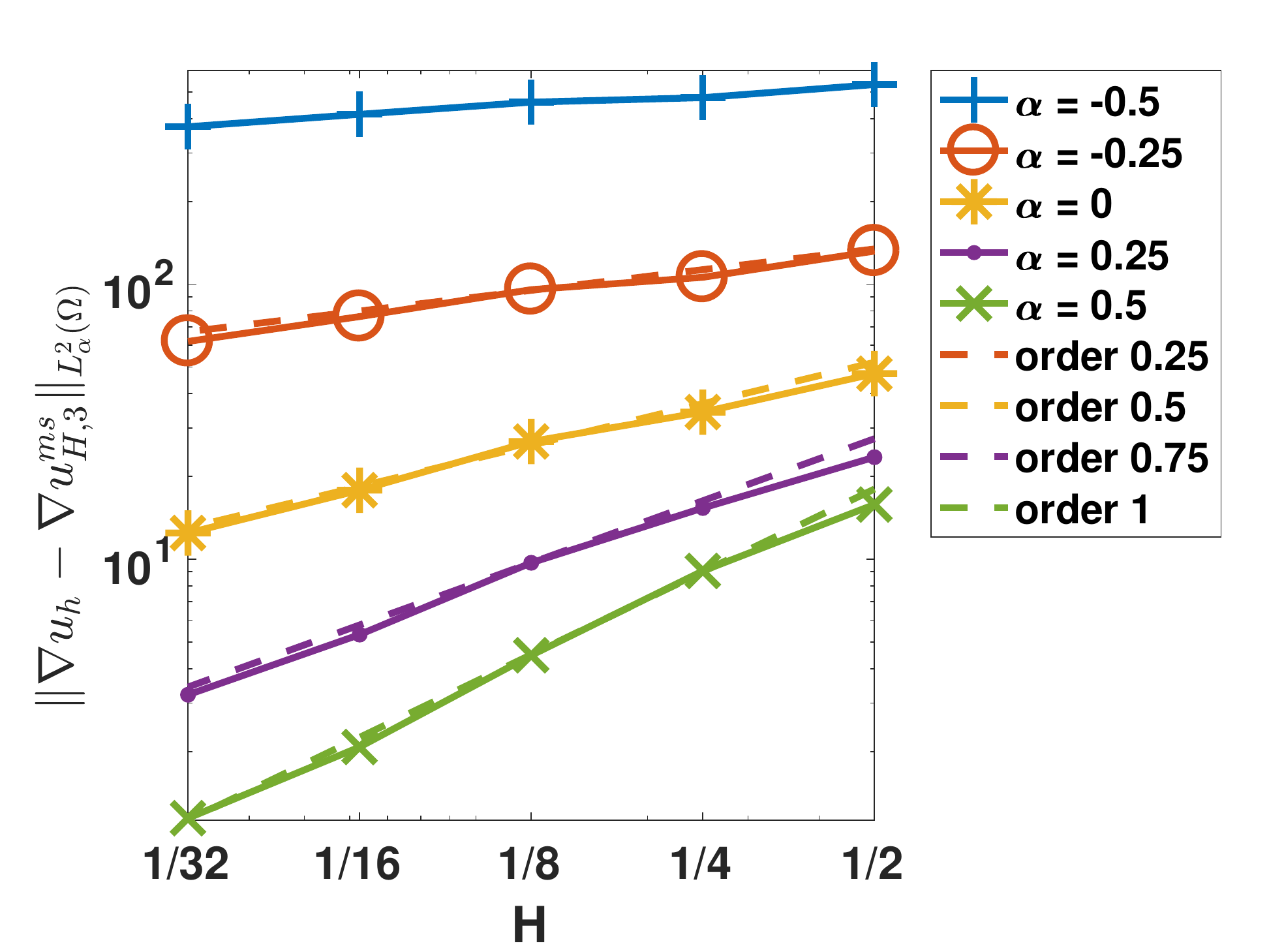}
\caption{Convergence histories for the highly oscillatory example with point source (left), point sink together with point source (middle), and line source (right).}
\label{fig:sinus}
\end{figure}
In this example we consider a highly oscillatory permeability
\[
  A(x) = 1+\frac{1}{4}\left(\sin\left(\frac{\pi x_1}{2^5}\right)+\sin\left(\frac{\pi x_2}{2^5}  \right)\right),
\]
with values between $1/2$ and $3/2$.
Note that none of the coarse meshes with mesh size $H=2^{-1},\ldots,2^{-5}$ resolves the oscillations.
In Figure~\ref{fig:sinus}, we present the convergence results for the three different singular source terms.
For the point singularity with a single source or two point singularities with a sink and a source, we observe 
$\mathcal{O}(H^\alpha)$ convergence of the error $\lVert\nabla u_h-\nabla u_{H,3}^{ms}\rVert_{L^2_\alpha(\Omega)}$
for $\alpha = \frac{1}{4},\frac{1}{2},\frac{3}{4},1$.
For the line fracture we observe the order of convergence $\mathcal{O}(H^{\alpha+1/2})$
for $\alpha = -\frac{1}{4},0,\frac{1}{4},\frac{1}{2}$.
These results confirm the theoretical convergence rates of Theorem~\ref{errorlocal} and
show that the convergence is independent of the highly oscillating coefficient.

\subsection{Heterogeneous Example}
\begin{figure}
\centering
\includegraphics[width=0.32\textwidth]{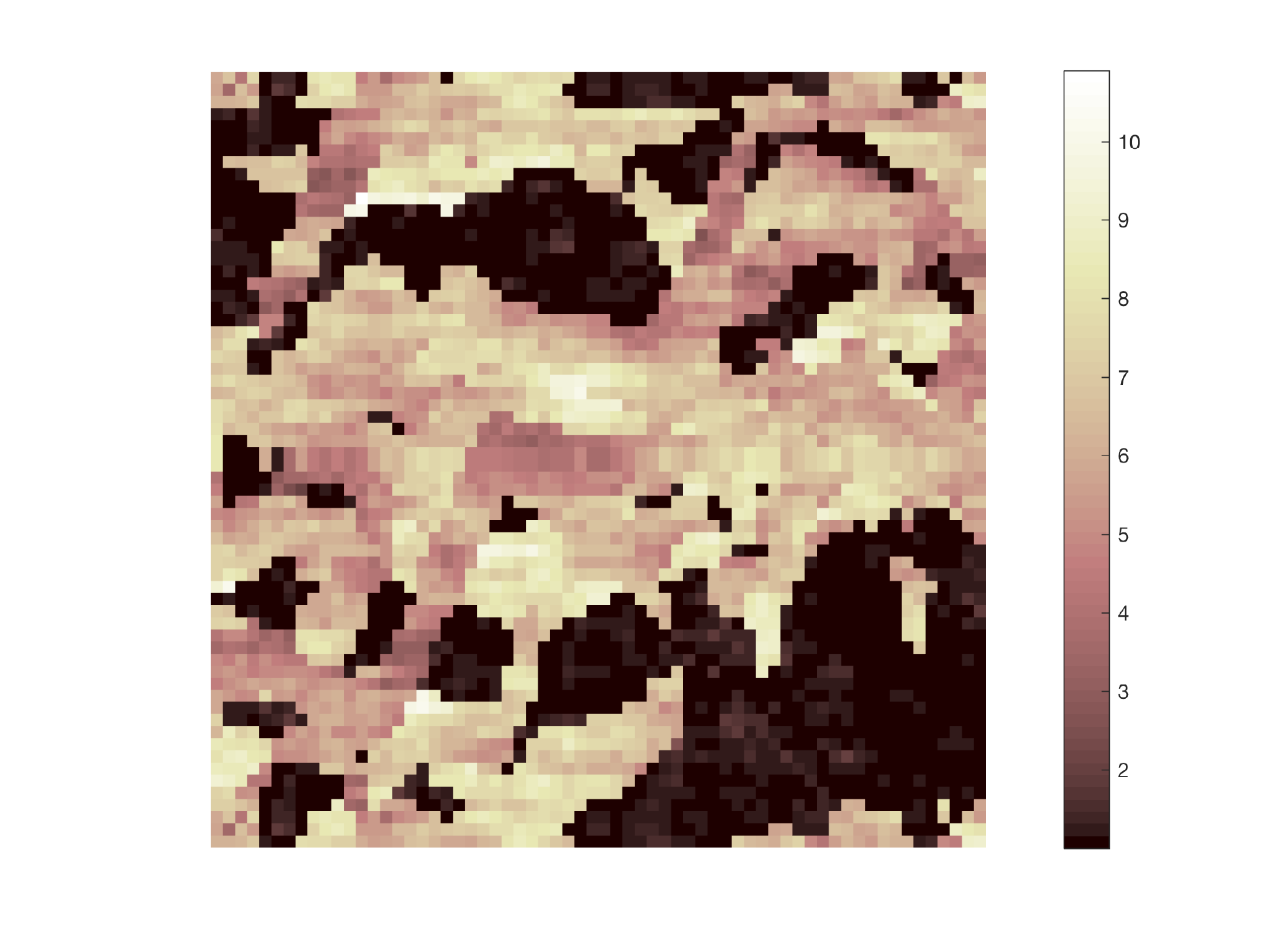}
\caption{Permeability in the heterogeneous example with scaled data from the SPE10 data.}
\label{fig:permeability}
\end{figure}
\begin{figure}
\centering
\includegraphics[width=0.32\textwidth]{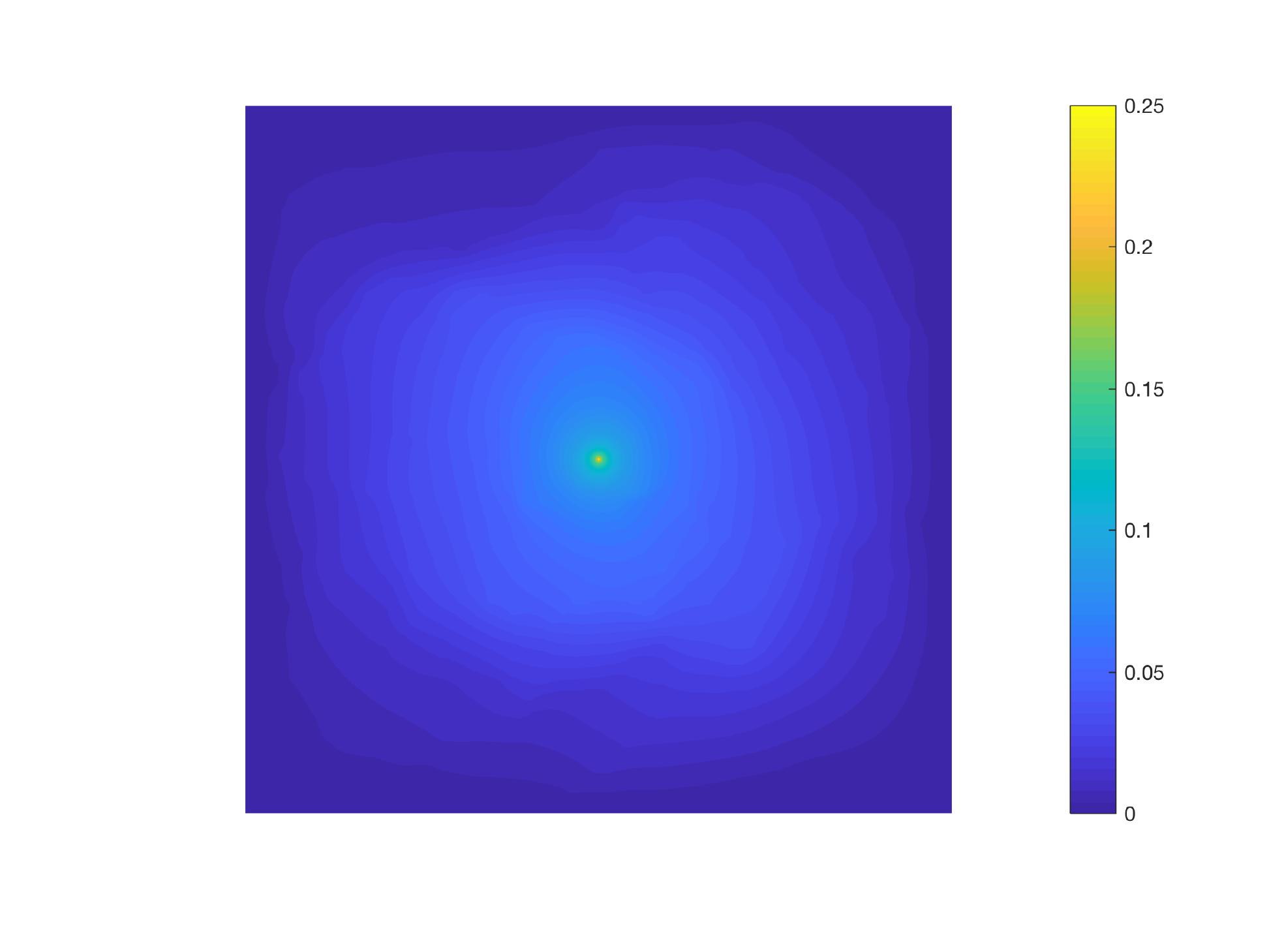}
\includegraphics[width=0.32\textwidth]{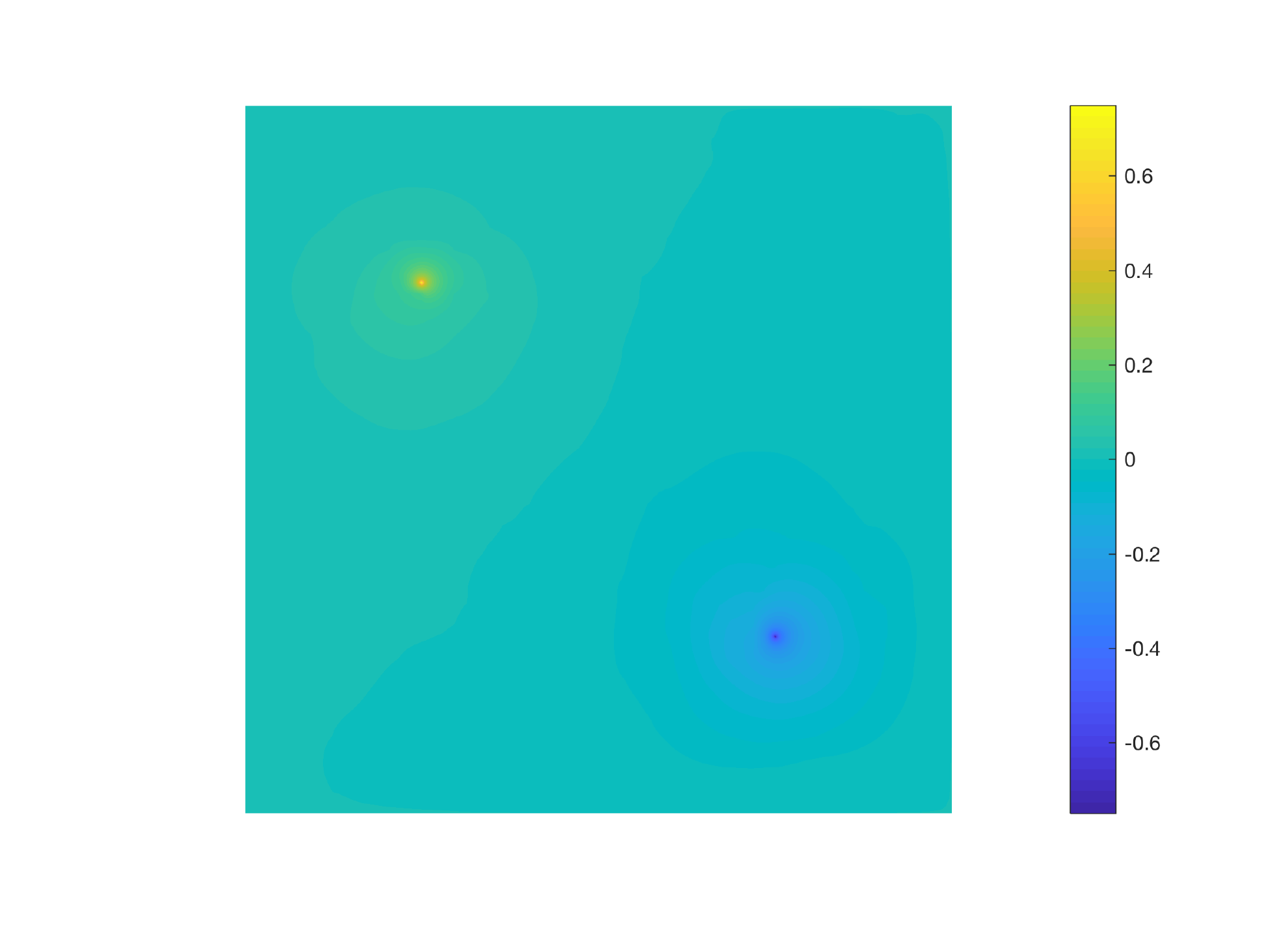}
\includegraphics[width=0.32\textwidth]{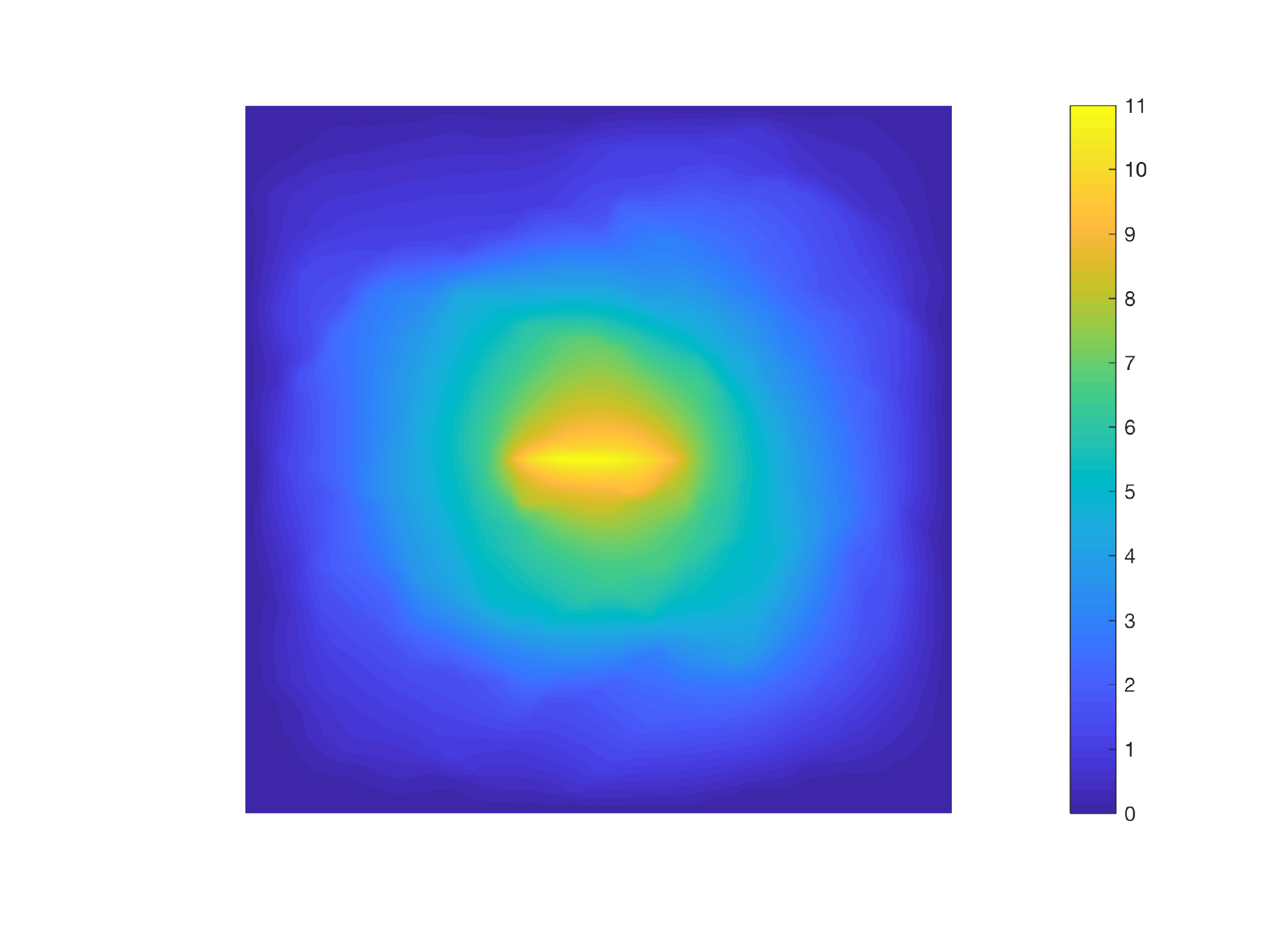}\\
\includegraphics[width=0.32\textwidth]{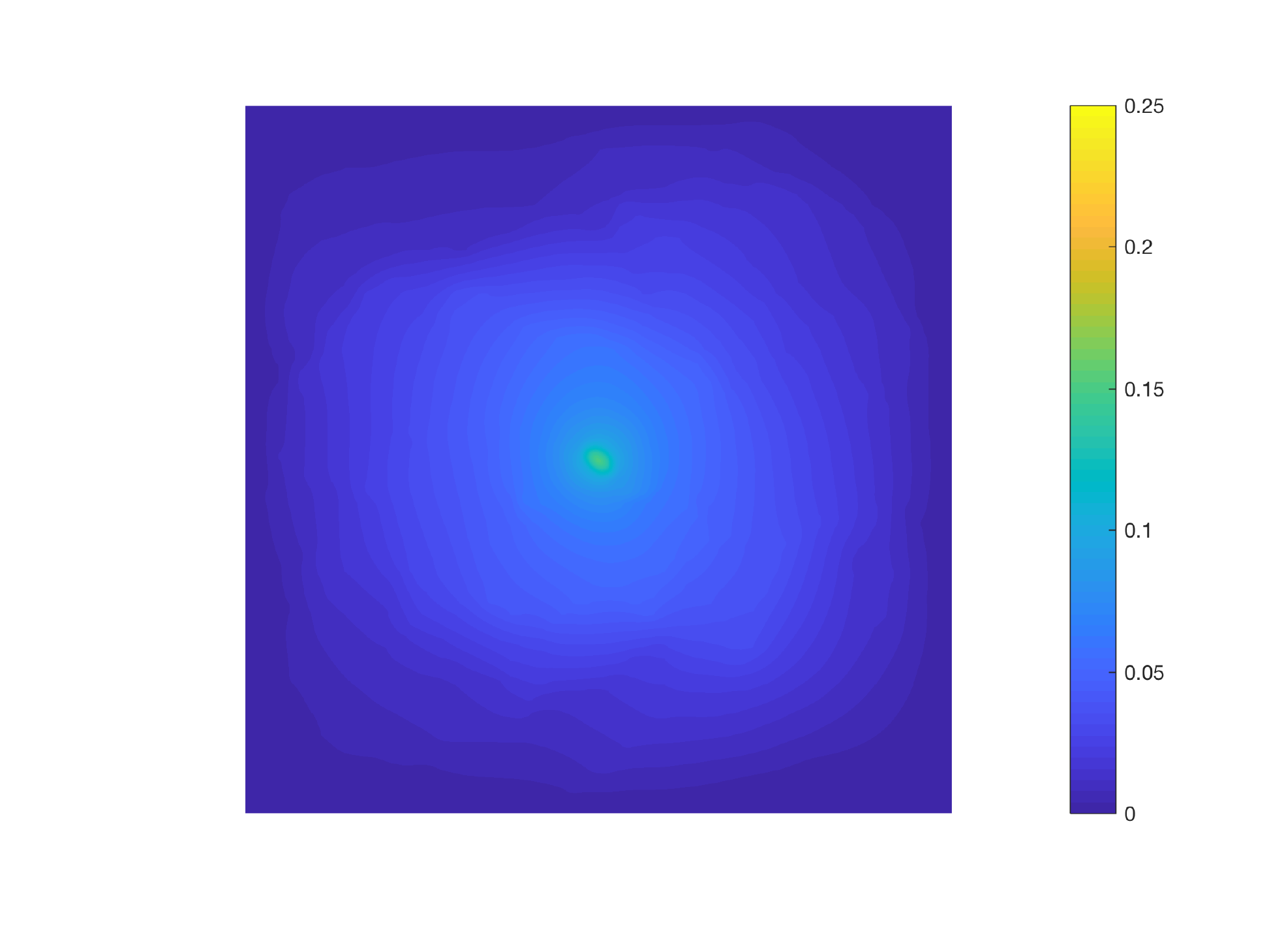}
\includegraphics[width=0.32\textwidth]{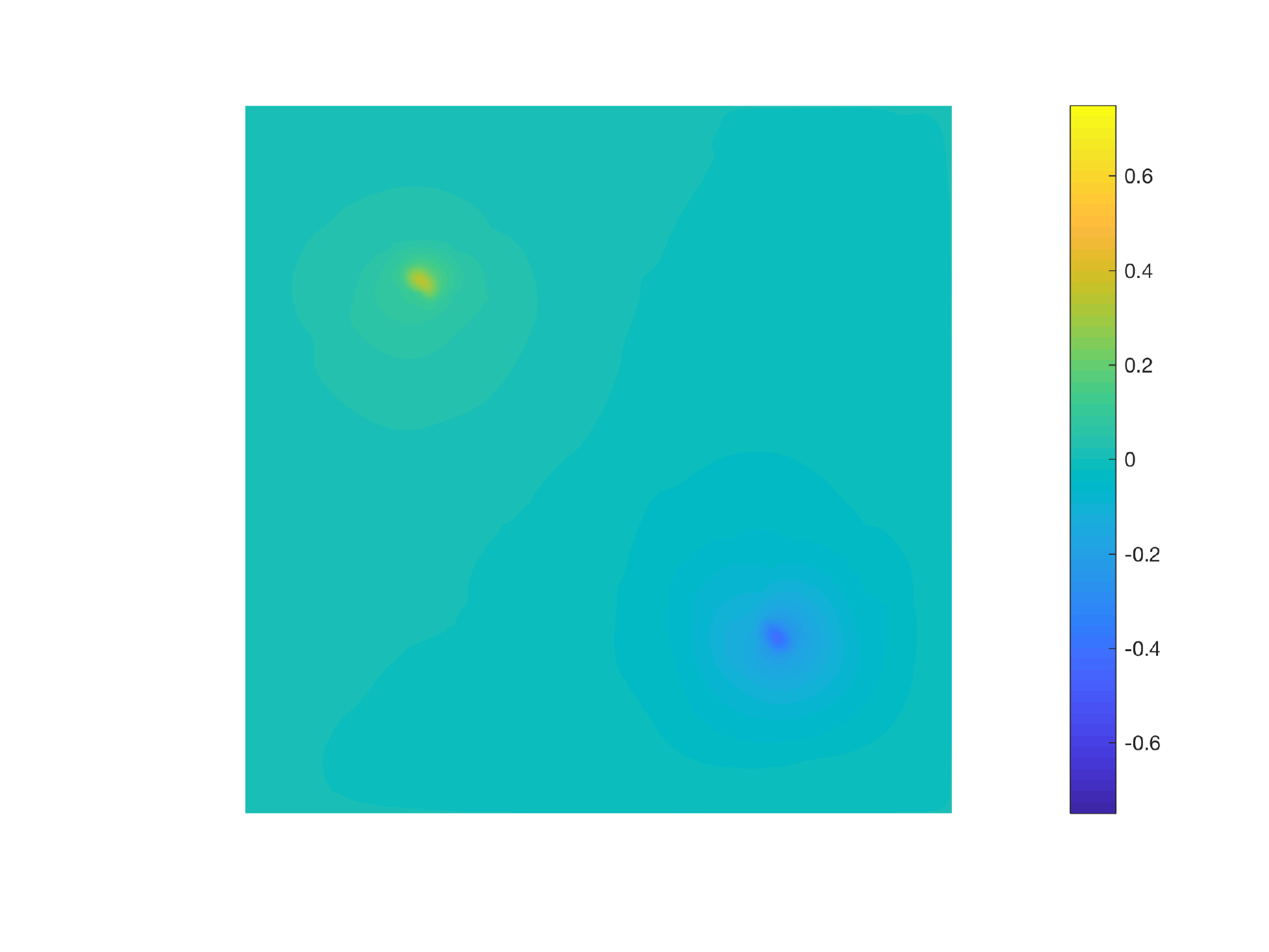}
\includegraphics[width=0.32\textwidth]{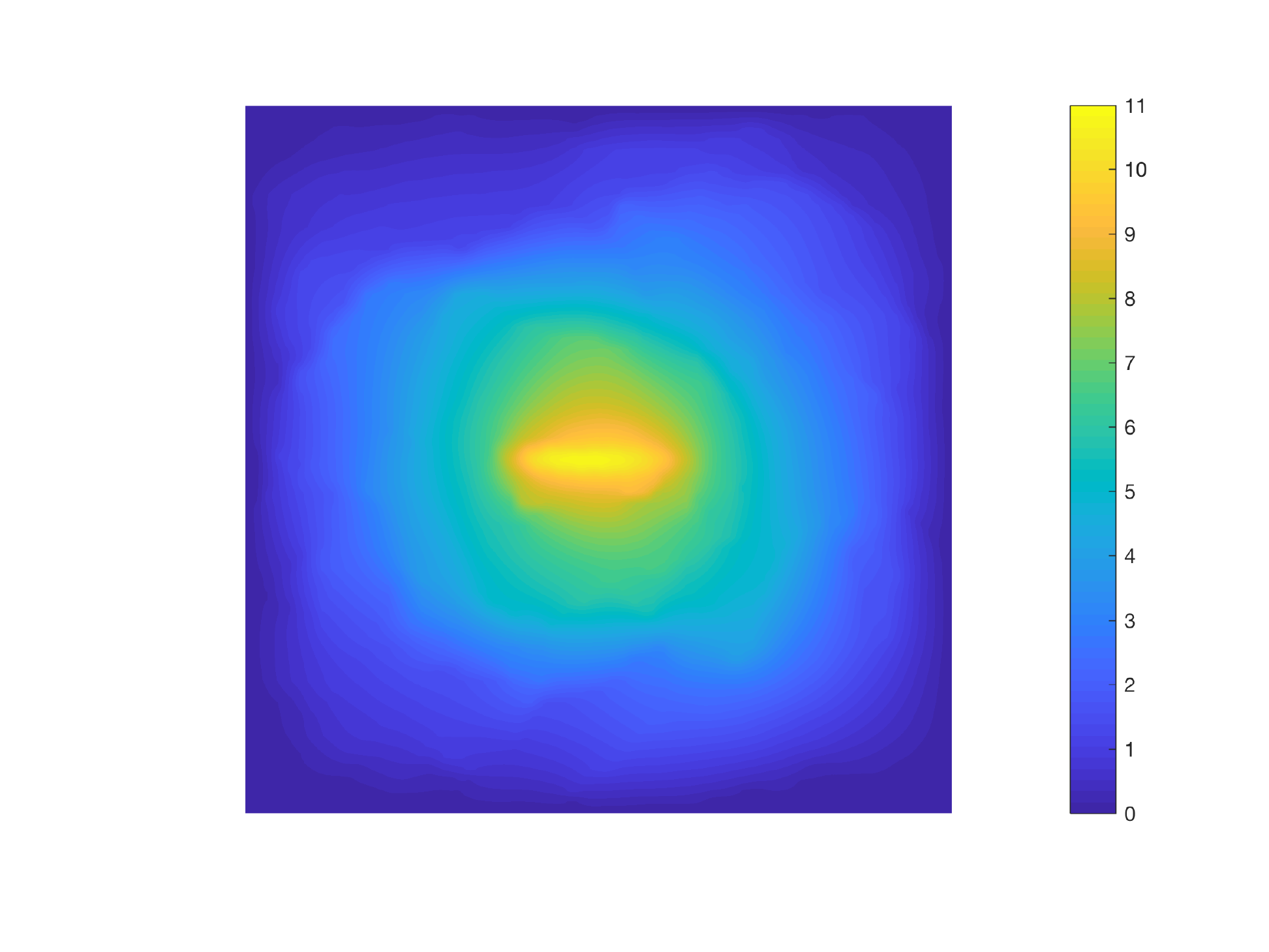}\\
\caption{Approximations for the heterogeneous example with point source for $\alpha=0.5$ (left), point sink together with point source for $\alpha=0.5$ (middle), and line fracture for $\alpha=0$ (right). Fine scale approximation with $h=2^{-9}$ (top) compared to the LOD approximation with $H=2^{-5}$ (bottom).}
\label{fig:heterogeneous:approximations}
\end{figure}
\begin{figure}
\centering
\includegraphics[width=0.32\textwidth]{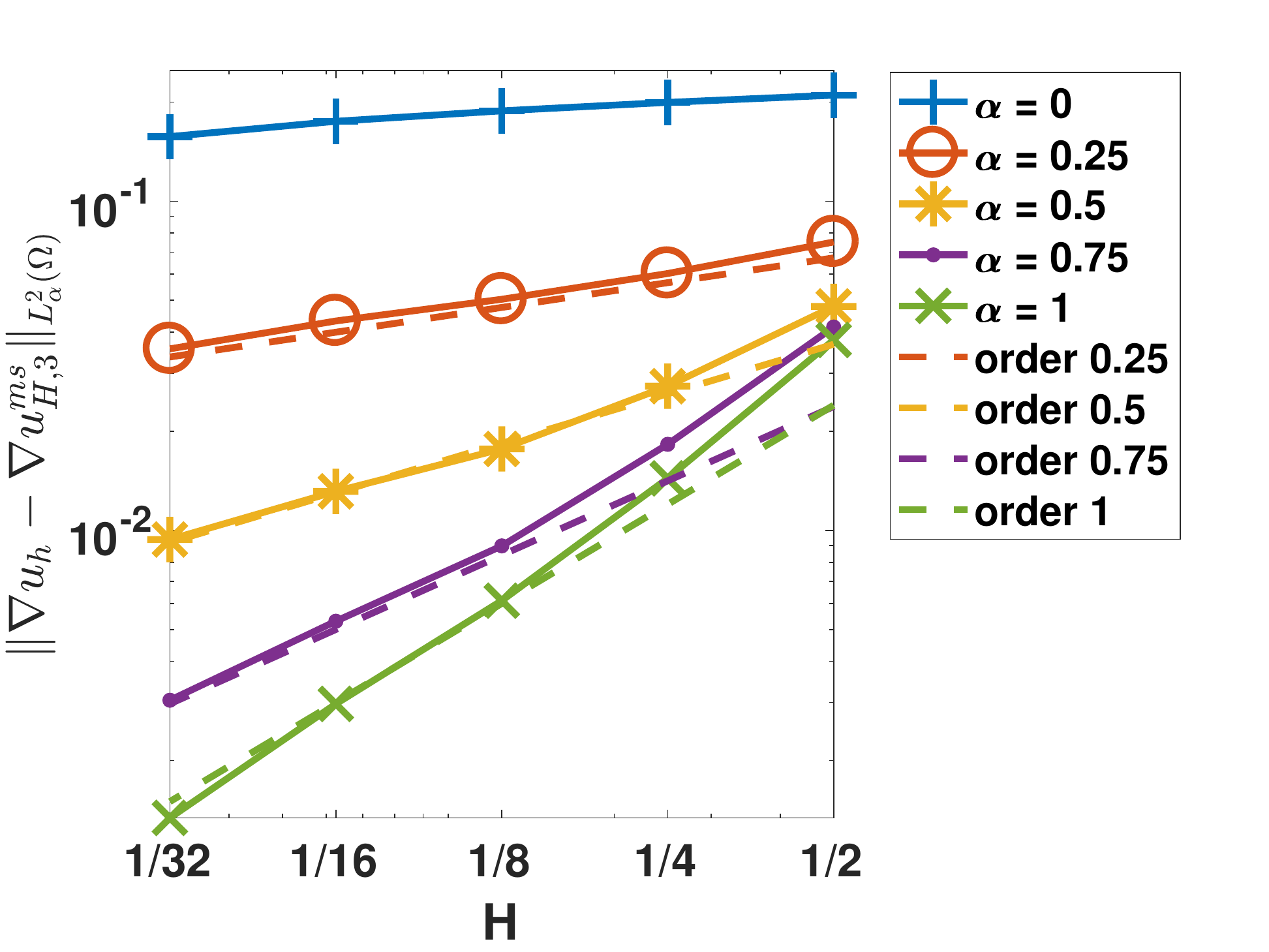}
\includegraphics[width=0.32\textwidth]{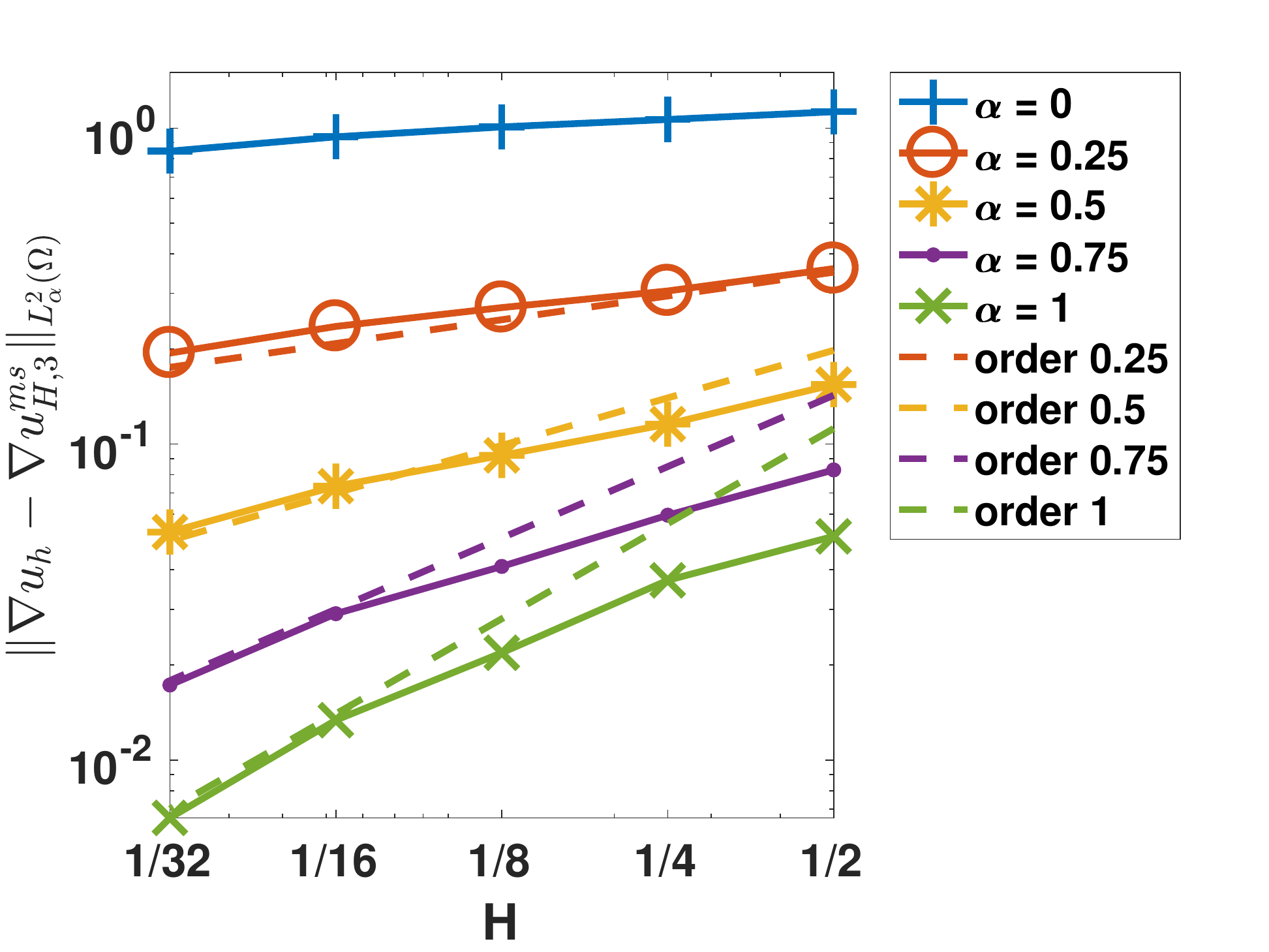}
\includegraphics[width=0.32\textwidth]{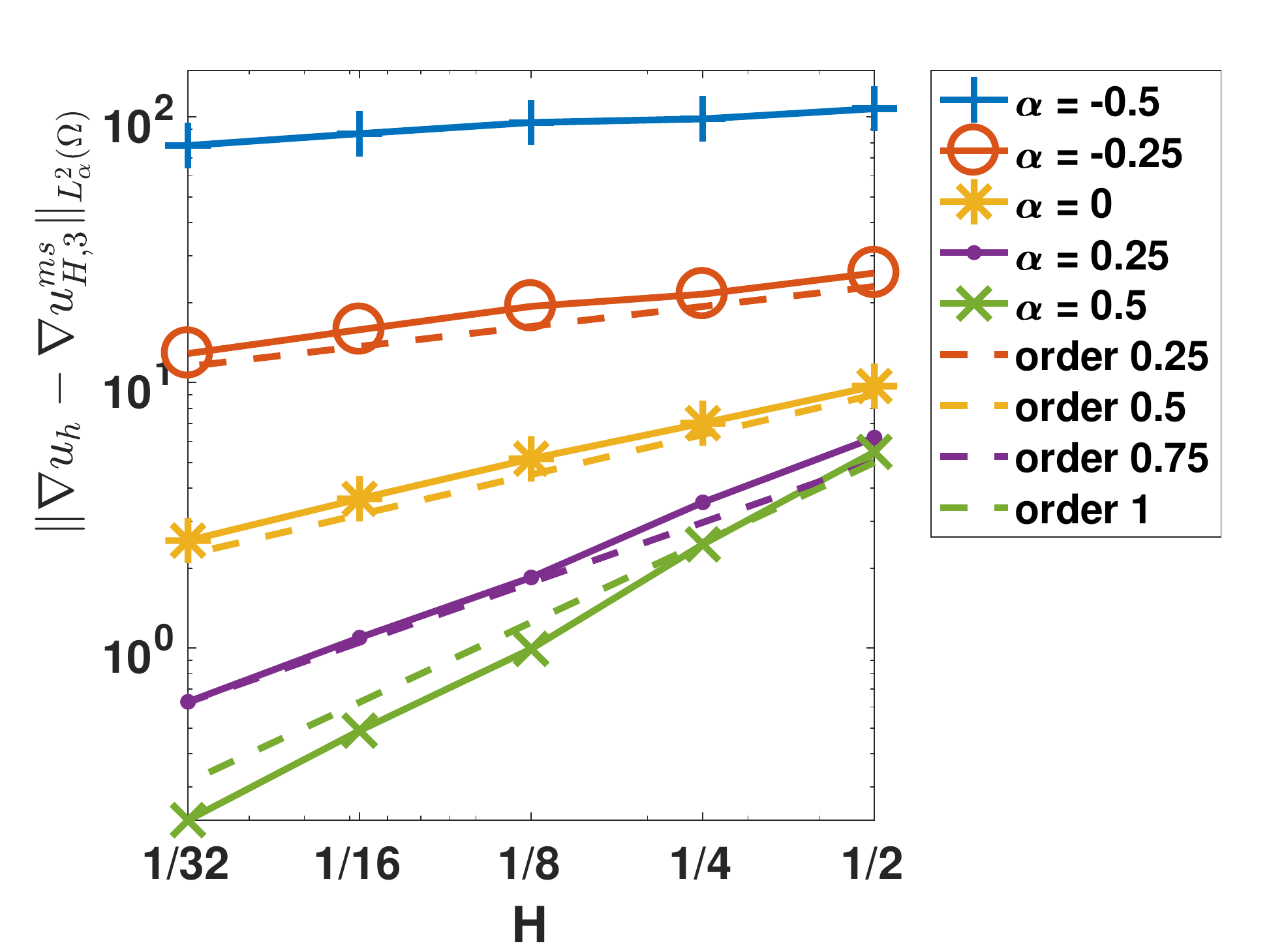}
\caption{Convergence histories for the heterogeneous example with point source (left), point sink together with point source (middle), and line source (right).}
\label{fig:heterogeneous}
\end{figure}
In this example we choose a permeability $A$ without any (periodic) structure
to demonstrate the generality of the method.
We consider the permeability of Figure~\ref{fig:permeability} with values between $1$ and $11$, taken from the SPE10 data, which has been
rescaled with the function $1+\log(1+z)$ in order to reduce the high contrast of the data.
Note that the theory here does not prevent issues from high contrast coefficients, and these ratios of material properties maybe tracked in the analysis.
Still, in unreported numerical experiments we observe convergence of the LOD method
for singular sources and the original high contrast data, at the cost of slower or in some cases even faster convergence rates
than predicted by the theory, which are arguably pre-asymptotic.
Since high contrast is not in the focus of this paper,
we reduced the contrast, in order to demonstrate the theoretical convergence rates for very coarse meshes.

In Figure~\ref{fig:heterogeneous:approximations} we display the fine scale FEM approximations
together with the multiscale approximations on refinement level 5.
We observe that the multiscale approximations resemble the fine scale features of the fine scale approximation very well.
In Figure~\ref{fig:heterogeneous} we observe convergence of $\mathcal{O}(H^\alpha)$
of the error $\lVert\nabla u_h-\nabla u_{H,3}^{ms}\rVert_{L^2_\alpha(\Omega)}$
for the point singular source terms
and $\mathcal{O}(H^{\alpha+1/2})$ for the singular line fracture, which confirms Theorem~\ref{errorlocal}
in the case of an unstructured permeability $A$ with moderate contrast.

\section{Acknowledgments}
The second author has been funded by the Austrian Science Fund (FWF) through the
project P 29197-N32.
Main parts of this paper were written while the authors enjoyed
the kind hospitality of the Hausdorff Institute for Mathematics in Bonn during the trimester program on 
multiscale problems in 2017. The first author would like to acknowledge Kris van der Zee for the discussions on Banach spaces and partial differential equations which lead to the first author looking into singular source problems.

\appendix 

\section{Quasi-Interpolation Stability }\label{stabAppendix}

Here we present the proof of \eqref{stableoperator} and \eqref{stableproj} from Lemma \ref{stablelemma}.

\begin{proof}[Proof of Lemma \ref{stablelemma} ]
	Suppose that $\vnode'\in {\cal N}(\omega_\vnode)$. If $\vnode'\in {\cal N}_{int}(\omega_\vnode)$, then 
	noting that ${\cal P}_{\vnode'}u$ is finite dimensional and  using  
	the following result from classical finite element  inverse inequalities
	\[
	\norm{{\cal P}_{\vnode'}u}_{L^r(\omega_{\vnode'}) }\lesssim  |\omega_{\vnode'}|^{\left(\frac{1}{r}-\frac{1}{s}\right)}\norm{{\cal P}_{\vnode'}u}_{L^s (\omega_{\vnode'} ) }, \text{ for } 1\leq s\leq r<\infty,
	\]
	for $r=\infty,s=1$, we obtain 
	\begin{align*}
	\norm{{\cal P}_{\vnode'}u}_{L^\infty(\omega_{\vnode'}) }&\lesssim  |\omega_{\vnode'}|^{-1}\norm{{\cal P}_{\vnode'}u}_{L^1 (\omega_{\vnode'}) }=|\omega_{\vnode'}|^{-1}\int_{\omega_{\vnode'}}|{\cal P}_{\vnode'}u|({d}^{2\beta}_{\Lambda})^{\frac{1}{2}} (\dap)^{-\frac{1}{2}} \,dx\\
	&\leq |\omega_{\vnode'}|^{-1}  \TwoNorm{{\cal P}_{\vnode'}u}{\omega_{\vnode'}}  \left(\int_{\omega_{\vnode'}} \dam \, dx\right)^{\frac{1}{2}}.
	\end{align*}
	Here we 	use  the obvious notation $ \norm{\cdot }^s_{L^s_{\beta}(\omega_\vnode')}=\int_{\omega_\vnode'} (\cdot)^s \dap dx$, with $s\in [1,\infty)$.
	From \eqref{L2proj}, letting $v_H=({\cal P}_{\vnode'} u)$, 
	 {
	\begin{align*}
	\norm{ {\cal P}_{\vnode'} u}^2_{L_{\beta}^2(\omega_{\vnode'} ) }&=\int_{{\omega}_{\vnode'}} |{\cal P}_{\vnode'}u|^2 \dap \, dx= \int_{{\omega}_{\vnode'}}u {\cal P}_{\vnode'}u \, \dap \,  dx\\
	&\lesssim \norm{ {\cal P}_{\vnode'} u}_{L^{\infty}(\omega_{\vnode'})}\int_{\omega_{\vnode'}} |u| (\dap)^{\frac{1}{2}} (\dap)^{\frac{1}{2}}  dx\\
	&\lesssim  \norm{ {\cal P}_{\vnode'} u}_{L^{\infty}(\omega_{\vnode'})} \norm{u }_{L^2_{\beta}(\omega_\vnode')}	\left(\int_{\omega_{\vnode'}} \dap dx\right)^{\frac{1}{2}}.
	\end{align*}
	}
	%
	Thus,  manipulating the  above identities 
	 {
	\begin{align*}
	\norm{ {\cal P}_{\vnode'} u}^2_{L^{\infty}(\omega_{\vnode'})} 
	&\lesssim |\omega_{\vnode'}|^{-2}  \left(\int_{\omega_{\vnode'}}\dam \, dx\right)^{}  \left(\int_{\omega_{\vnode'}} \dap dx\right)^{\frac{1}{2}}\norm{u }_{L_{\beta}^2(\omega_{\vnode'})} 	  \norm{ {\cal P}_{\vnode'} u}_{L^{\infty}(\omega_{\vnode'})} \nonumber.
	\end{align*}
	Rearranging terms
	and by taking the larger patch $\omega_{\vnode,1} \supset\omega_{\vnode'}$, we have 
	\begin{align}	\label{nodeinfty}
	\left | {\cal P}_{\vnode'} u(\vnode')\right| &\lesssim  |\omega_{\vnode,1}|^{-2}  \left(\int_{\omega_{\vnode,1}}\dam \, dx\right)^{}  \left(\int_{\omega_{\vnode,1}} \dap dx\right)^{\frac{1}{2}}\norm{u }_{L^2_{\beta}(\omega_{\vnode,1})}.	  
	\end{align}
	}
	%
	Finally, we note  (again taking a larger domain $\omega_{\vnode,1}$ to $\omega_{\vnode}$)  that
	\begin{align}\label{basisfunctionweighted}
	\TwoNorm{\lambda_{\vnode'}}{\omega_{\vnode}}\lesssim \left(\int_{\omega_{\vnode,1}}\dap \, dx\right)^{\frac{1}{2}}, 
	\qquad\text{and}\qquad
	 \TwoNorm{\nabla \lambda_{\vnode'}}{\omega_{\vnode}}\lesssim H^{-1} \left(\int_{\omega_{\vnode,1}}\dap \, dx\right)^{\frac{1}{2}}.
	\end{align}
	For the quasi-interpolation $\Qint{\beta}(u)$ we have 
	\[
	\Qint{\beta}(u)=\sum_{\vnode'\in{\cal N}_{int}(\omega_\vnode)} ({\cal P}_{\vnode'} u)(\vnode') \lambda_{\vnode'}  \text{ in }\omega_\vnode.
	\]
	For $L^2$ stability, we note that from \eqref{nodeinfty} and \eqref{basisfunctionweighted}, we obtain
	 {
	\begin{align}\label{stability.QI1}
	\TwoNorm{\Qint{\beta}(u)}{\omega_\vnode}&\leq \sum_{\vnode'\in{\cal N}_{int}(\omega_\vnode)} \left|({\cal P}_{\vnode'} u)(\vnode')\right| \TwoNorm{ \lambda_{\vnode'}}{\omega_{\vnode}} \nonumber \\
	&\lesssim\left( \frac{ |B|}{|\omega_{\vnode,1}|} \right)^2\frac{1}{|B|^2}  \left(\int_{B}\dam \, dx\right)^{}  \left(\int_{B} \dap dx\right)^{ }\norm{u }_{L_{\beta}^2(\omega_{\vnode,1})}\nonumber\\
	&\lesssim  	 \left( \frac{ |B|}{|\omega_{\vnode,1}|} \right)^2 C_{2,\beta} \norm{u }_{L_{\beta}^2(\omega_{\vnode,1})},
	\end{align} }
	where we used the Muckenhoupt weight condition from  Proposition \ref{alphaMuckenhoupt}. 
	 {
	We  take $B$ to be the ball containing the patch $\omega_{\vnode,1}$, and we
	}
	suppose (by quasi-uniformity) that the ratio $\left( \frac{ |B|}{|\omega_{\vnode,1}|} \right)$ is trivially bounded. 
	\par
	For the $H^1$ stability, first noting that $\avrg{u}{\omega_{\vnode,1}}=\Qint{\beta}(\avrg{u}{\omega_{\vnode,1}})$, we denote $\bar{u}=u-\avrg{u}{\omega_{\vnode,1}}$.
	Thus, from  \eqref{nodeinfty} and \eqref{basisfunctionweighted}, and arguments used above, we obtain 
	 {
	\begin{align}\label{stability.QI2}
	&\TwoNorm{\nabla \Qint{\beta}(u)}{\omega_{\vnode}} =\TwoNorm{\nabla \Qint{\beta}(\bar{u})}{\omega_{\vnode}} \lesssim \sum_{\vnode'\in{\cal N}_{int}(\omega_\vnode)} \left|({\cal P}_{\vnode'} \bar{u})(\vnode')\right| \TwoNorm{ \nabla \lambda_{\vnode'}}{\omega_{\vnode}} \nonumber \\
	&\lesssim H^{-1} |\omega_{\vnode,1}|^{-2}  \left(\int_{\omega_{\vnode,1}}\dam\, dx\right)^{}  \left(\int_{\omega_{\vnode,1}} \dap dx\right)^{}\norm{\bar{u} }_{L_{\beta}^2(\omega_{\vnode,1})}   \lesssim   	 \left( \frac{ |B|}{|\omega_{\vnode,1}|} \right)^2 C_{2,\beta}\norm{\nabla {u} }_{L_{\beta}^2(\omega_{\vnode,1})}, 
	\end{align} }
	where for the last inequality we used the weighted Poincar\'{e} inequality from Lemma \ref{poincare2}.
	To prove local $L^2$ approximability, we note that for $\bar{u}=u-\avrg{u}{\omega_{\vnode,1}}$, using \eqref{stability.QI1} and Lemma \ref{poincare2}, we obtain 
	 {
	\begin{align}\label{L2StableProofderp}
	\TwoNorm{u-\Qint{\beta}(u)}{\omega_\vnode}&=\TwoNorm{\bar{u}-\Qint{\beta}(\bar{u})}{\omega_\vnode} \leq \TwoNorm{\bar{u} }{\omega_{\vnode}}+\TwoNorm{ \Qint{\beta}(\bar{u})}{\omega_\vnode}\\ \nonumber
	&\lesssim \TwoNorm{\bar{u} }{\omega_{\vnode,1}}
	 \lesssim H\TwoNorm{ \nabla {u}}{\omega_{\vnode,1}}.
	\end{align} }
Thus, local approximability holds, and result \eqref{stableproj2} trivially holds from $H^1$ stability. From arguments in \cite{brown2016multiscale}, we deduce that $\Qint{\beta}$ is also a projection. 
	\qed
\end{proof}

\section{Truncation Estimates}\label{truncproofsection}
Now we will prove and state the auxiliary lemmas used to prove the localized error estimate in Theorem \ref{errorlocal}.  These proofs are entirely based on the works \cite{brown2017numerical,Henning.Morgenstern.Peterseim:2014,MP11} and references therein. The proofs have been extended to weighted spaces in \cite{brown2017numerical}. In that work the weight function was $y^\gamma$, for some $\gamma\in (-1,1)$, we replace this weight with $\dap$ and will have a very similar proof.
Here we present a  version of these ideas and highlight any subtle differences.  
\par

We begin with defining some standard technical cutoff functions. For $\vnode,\vnode' \in{\cal N}_{int}$ and $l,k\in \mathbb{N}$ and $m=0,1,\ldots$, with $k\geq l\geq 2$ we have
\begin{align}\label{quasiinclusion}
\text{if}\qquad{\omega}_{\vnode',m+1}\cap \left({\omega}_{\vnode,k}\backslash {\omega}_{\vnode,l} \right)\neq \emptyset, \qquad\text{then}\qquad{\omega}_{\vnode',1}\subset\left({\omega}_{\vnode,k+m+1}\backslash {\omega}_{\vnode,l-m-1} \right).
\end{align}
We will use the cutoff functions defined in \cite{Henning.Morgenstern.Peterseim:2014}. 
For $\vnode\in {\cal N}_{int}$ and $k>l\in \mathbb{N}$, let $\eta^{k,l}_{\vnode}:\dmn \to [0,1]$ be a continuous weakly differentiable function so that 
\begin{subequations}\label{cutoff1}
	\begin{align}
	\left(   \eta^{k,l}_{\vnode}  \right)|_{{\omega}_{\vnode,k-l}}&=0,\\
	\left(   \eta^{k,l}_{\vnode}  \right)|_{\Omega \backslash {\omega}_{\vnode,k}}&=1,\\
	\forall T\in {\cal T}_{\Omega}, \norm{\nabla \eta^{k,l}_\vnode}_{L^{\infty}(T)}&\leq C_{co}\frac{1}{l H_{}}, 
	\end{align}
\end{subequations}
where $C_{co}$ is only dependent on the shape regularity of the mesh ${\cal T}_{H}$.
We  choose here the cutoff function as in \cite{MP11} where we choose a function $\eta^{k,l}_{\vnode}$ in the space of $\mathbb{P}_{1}$ Lagrange finite elements over ${\cal  T}_{H}$ such that 
\begin{align*}
\eta^{k,l}_{\vnode}(\vnode')&=0 \text{ for all } \vnode' \in {\cal N}_{int}\cap \omega_{\vnode,k-l},\nl
\eta^{k,l}_{\vnode}(\vnode')&=1 \text{ for all } \vnode'\in {\cal N}_{int}\cap (\Omega\backslash \omega_{\vnode,k}),\nl
\eta^{k,l}_{\vnode}(\vnode')&=\frac{j}{l} \text{ for all } \vnode'\in {\cal N}_{int}\cap \omega_{\vnode,k-l+j}, j=0,1, \dots ,l.
\end{align*}
\par
 
We will now prove a lemma showing the quasi-invariance of the fine-scale functions under multiplication by cutoff functions in the distance-weighted Sobolev space.
 \begin{lemma}\label{qi}
 	Let $k>l\in \mathbb{N}$,  $\vnode \in {\cal N}_{int}$, and $\beta\in \left(-\frac{d-\ell}{2},\frac{d-\ell}{2} \right)$. Suppose that $w\in {V}^f_{\beta}$, then we have the estimate
 	\begin{align*}
 	\TwoNorm{\nabla \Qint{\beta}(\eta_{\vnode}^{k,l} w)   }{\dmn}\lesssim l^{-1}\TwoNorm{ \nabla w }{ {\omega}_{\vnode,k+2}\backslash {\omega}_{\vnode,k-l-2} }.
 	\end{align*}
 \end{lemma}
\texttt{}
\begin{proof}
	Fix $\vnode$ and $k$,  and denote the average  as $\avrg{\eta_{\vnode}^{k,l}}{{\omega}_{\vnode',1} }=\frac{1}{|{\omega}_{\vnode',1}|} \int_{{\omega}_{\vnode',1}} \eta_{\vnode }^{k,l}dx. $
	 For an estimate on a single patch ${\omega}_{\vnode'}$, using the stability \eqref{stableoperator} and the fact that $\Qint{\beta}(w)=0$,  we have 
	 \begin{align*}
	 \TwoNorm{\nabla \Qint{\beta}(\eta_{\vnode}^{k,l}w )}{{\omega}_{\vnode'}}&=\TwoNorm{\nabla \Qint{\beta}((\eta_{\vnode}^{k,l}-\avrg{\eta_{\vnode}^{k,l}}{{\omega}_{\vnode',1}}) w)}{{\omega}_{\vnode'}}\\
	 &\lesssim     \TwoNorm{(\eta_{\vnode}^{k,l}-\avrg{\eta_{\vnode}^{k,l}}{{\omega}_{\vnode',1}})\nabla w }{ {\omega}_{\vnode',1} }    +   \TwoNorm{\nabla  \eta^{k,l}_{\vnode} (w-\Qint{\beta}(w)) }{ {\omega}_{\vnode',1}} .
	 \end{align*}
	 Summing over all $\vnode'\in {\cal N}_{int}$, using the quasi-inclusion property \eqref{quasiinclusion} yields 
 \begin{align}
 \TwoNorm{\nabla \Qint{\beta}(\eta_{\vnode}^{k,l} w)   }{\cilt}^2 
 &\lesssim  \sum_{{\omega}_{\vnode',1}\subset{\omega}_{\vnode,k+1}\backslash {\omega}_{\vnode,k-l-1} }      \TwoNorm{ (\eta_{\vnode}^{k,l}-\avrg{\eta_{\vnode}^{k,l}}{{\omega}_{\vnode',1}})\nabla w }{ {\omega}_{\vnode',1} }^2    \nl
 \label{summingover}
 &+  \sum_{{\omega}_{\vnode',1}\subset{\omega}_{\vnode,k+1}\backslash {\omega}_{\vnode,k-l-1} }     \TwoNorm{\nabla  \eta^{k,l}_{\vnode}( w-\Qint{\beta}(w)) }{ {\omega}_{\vnode',1}}^2  .
 \end{align}
	 Here we used that $\nabla \eta_{\vnode}^{k,l}\neq 0$ only in ${\omega}_{\vnode,k}\backslash{\omega}_{\vnode,k-l} $ and 
	 $(\eta_{\vnode}^{k,l}-\avrg{\eta_{\vnode}^{k,l}}{{\omega}_{\vnode',1}})|_{{\omega}_{\vnode',1}}\neq0$ only if ${\omega}_{\vnode',1}$ intersects ${\omega}_{\vnode,k}\backslash{\omega}_{\vnode,k-l} $.
	 \par
	 
We now denote $\mu_{\vnode}^{k,l}=\eta_{\vnode}^{k,l}-\avrg{\eta_{\vnode}^{k,l}}{{\omega}_{\vnode',1}}$, 
and let ${T}$ be a  simplex in ${{\omega}_{\vnode',1}}$
such that the supremum $\norm{\mu_{\vnode}^{k,l}}_{L^{\infty}({\omega}_{\vnode',1})}$ is obtained.  
On ${T}$, $\mu_{\vnode}^{k,l}$ is an affine function, 
using the  fact that $\eta_{\vnode}^{k,l}$ is taken to be $\mathbb{P}_{1}$, we use  the following inverse estimate combined with the Muckenhoupt property Proposition \ref{alphaMuckenhoupt}. Note that  by 
	utilizing the following result from classical finite element  inverse inequalities
	\[
	\norm{q}_{L^r(T) }\lesssim  |T|^{\left(\frac{1}{r}-\frac{1}{s}\right)}\norm{q}_{L^q (T ) }, \text{ for } 1\leq s\leq r<\infty,
	\]
	for $r=\infty,s=1$, we obtain 
	\begin{align*}
	\norm{q}_{L^\infty(T ) }&\lesssim  |T|^{-1}\norm{q}_{L^1 (T ) }=|T|^{-1}\int_{T}|q|({d}^{2\beta}_{\Lambda})^{\frac{1}{2}} (\dap)^{-\frac{1}{2}} \,dx
	\leq |T|^{-1}  \TwoNorm{q}{T}  \left(\int_{T} \dam \, dx\right)^{\frac{1}{2}}.
	\end{align*}
  So that  
	\begin{align*}
\norm{\mu_{\vnode}^{k,l}}_{L^{\infty}({\omega}_{\vnode',1})}=	\norm{\mu_{\vnode}^{k,l}}_{L^\infty(T ) }
&\lesssim  |T|^{-1}\TwoNorm{\mu_{\vnode}^{k,l}}{T}  \left(\int_{T} \dam \, dx\right)^{\frac{1}{2}}.
	\end{align*}
Using the  above estimate and taking  the whole patch, we see that
\begin{align}\label{Clip}
\norm{\eta_{\vnode}^{k,l}-\avrg{\eta_{\vnode}^{k,l}}{{\omega}_{\vnode',1}}}_{L^{\infty}({\omega}_{\vnode',1})}&\lesssim  |{\omega}_{\vnode',1}|^{-1} \left(\int_{{\omega}_{\vnode',1}} \dam \, dx\right)^{\frac{1}{2}} \TwoNorm{\eta_{\vnode}^{k,l}-\avrg{\eta_{\vnode}^{k,l}}{{\omega}_{\vnode',1}}}{{\omega}_{\vnode',1}}\nl
&\lesssim  |{\omega}_{\vnode',1}|^{-1} \left(\int_{{\omega}_{\vnode',1}} \dam \, dx\right)^{\frac{1}{2}}H\TwoNorm{\nabla \eta_{\vnode}^{k,l}}{{\omega}_{\vnode',1}}\nl
&\lesssim  |{\omega}_{\vnode',1}|^{-1} \left(\int_{{\omega}_{\vnode',1}} \dam \, dx\right)^{\frac{1}{2}} \left(\int_{{\omega}_{\vnode',1}} \dap \, dx\right)^{\frac{1}{2}} H  \norm{\nabla \eta_{\vnode}^{k,l}}_{L^{\infty}({\omega}_{\vnode',1})}\nl
&\lesssim (C_{2,\beta}^{\frac{1}{2}})H\norm{\nabla \eta_{\vnode}^{k,l} }_{L^{\infty}({\omega}_{\vnode',1})},
\end{align}
where we used the Muckenhoupt weight bound \eqref{Muckenhouptconstant}, as well as quasi-uniformity of the grid.
Returning to \eqref{summingover}, using the above relation on the first term and  the approximation property \eqref{stableproj} on the second term, we obtain
  {
 \begin{align*}
 \TwoNorm{\nabla \Qint{\beta}(\eta_{\vnode}^{k,l} w)   }{\dmn}^2 &\lesssim
 H^2 \norm{\nabla \eta_{\vnode}^{k,l} }_{L^{\infty}(\dmn)}^2  \TwoNorm{ \nabla w }{ {\omega}_{\vnode,k+1}\backslash {\omega}_{\vnode,k-l-1} }^2\\
 &+H^2 \norm{\nabla \eta_{\vnode}^{k,l} }_{L^{\infty}(\dmn)}^2  \TwoNorm{ \nabla w }{ {\omega}_{\vnode,k+2}\backslash {\omega}_{\vnode,k-l-2} }^2.
 \end{align*}
 }
 Finally,  we arrive at 
 \begin{align*}
 \TwoNorm{\nabla \Qint{\beta}(\eta_{\vnode}^{k,l} w)   }{\dmn}^2 \lesssim l^{-2}
 \TwoNorm{ \nabla w }{ {\omega}_{\vnode,k+2}\backslash {\omega}_{\vnode,k-l-2} }^2,
 \end{align*}
 where we used $ \norm{\nabla \eta_{\vnode}^{k,l} }_{L^{\infty}(\dmn)}^2\lesssim 1/(lH)^2$.
 \qed
\end{proof}
\par

For the distance-weighted Sobolev space, we have the following decay of the fine-scale space: 
 \begin{lemma}\label{decaylemma} Let $\beta\in \left(-\frac{d-\ell}{2},\frac{d-\ell}{2} \right)$.
 	Fix some $\vnode\in {\cal N}_{int}$ and $F\in ({V}_{\beta}^f)'$ the dual of ${V}_{\beta}^f$ satisfying $F(w)=0$ for all $w\in {V}^f_{\beta}(\dmn\backslash {\omega}_{\vnode,1})$.  Let $u\in {V}_{\beta}^f$ be the solution of 
 	\begin{align*}
 	\int_{\dmn}A \nabla u \nabla w \,   {\dap} dx =F(w) \quad\text{for all } w\in {V}_{\beta}^f.
 	\end{align*}
 	Then, there exists a constant $\theta\in (0,1)$ such that for $k\in \mathbb{N}$  we have 
 	\begin{align*}
 	\TwoNorm{\nabla u}{\dmn \backslash {\omega}_{\vnode,k}}\lesssim \theta^{k} \TwoNorm{\nabla u}{\dmn }.
 	\end{align*}
 \end{lemma}
\begin{proof}
	Let $\eta_{\vnode}^{k,l}$ be the cut-off function as in the previous lemma for $l<k-2$. Let $\tilde{u}=\eta_{\vnode}^{k,l} u -\Qint{\beta}(\eta_{\vnode}^{k,l} u )\in {V}_{\beta}^f(\dmn\backslash {\omega}_{\vnode,k-l-1})$, and note that from Lemma \ref{qi} we have 
	\begin{align}\label{qiestimate}
	\TwoNorm{\nabla( \eta_{\vnode}^{k,l}u-\tilde{u})}{\dmn }=\TwoNorm{\nabla \Qint{\beta}(\eta_{\vnode}^{k,l} u ) }{\dmn }\lesssim  l^{-1}\TwoNorm{ \nabla u }{ {\omega}_{\vnode,k+2}\backslash {\omega}_{\vnode,k-l-2} }.
	\end{align}
	From this estimate and the properties of $F$ we have 
	\begin{align}\label{relation}
	\int_{\dmn\backslash {\omega}_{\vnode,k-l-1}} A \nabla u \nabla \tilde{u}\, { \dap } dx=\int_{\dmn}A \nabla u \nabla \tilde{u} \,   { \dap } dx=F(\tilde{u})=0.
	\end{align}
{
We utilize a version of the Caccioppoli inequality \cite{caffarelli2016fractional} for the distance-weighted space, and the coercivity of the corrector problems \eqref{corrector} to obtain
\begin{align*}
\TwoNorm{\nabla u}{\dmn\backslash {\omega}_{\vnode,k}}^2&\lesssim \int_{\Omega\backslash {\omega}_{\vnode,k-l-1}} \eta_{\vnode}^{k,l} A\nabla u\nabla u\,  {\dap} \, dx 
=  \int_{\dmn\backslash {\omega}_{\vnode,k-l-1}} A\nabla u\left(\nabla (\eta_{\vnode}^{k,l} u ) -u\nabla \eta_{\vnode}^{k,l}     \right)\,  {\dap}  dx.
\end{align*}	
}
Using the fact that $\Qint{\beta}(u)=0$, estimate \eqref{qiestimate}, and  the relation \eqref{relation}, we have 
 {
\begin{align*}
\TwoNorm{\nabla u}{\dmn\backslash {\omega}_{\vnode,k}}^2 &\lesssim  \int_{\dmn\backslash {\omega}_{\vnode,k-l-1}} A\nabla u(\nabla (\eta_{\vnode}^{k,l} u -\tilde{u}))\,   {\dap } dx-\int_{\dmn\backslash {\omega}_{\vnode,k-l-1}}A\nabla u(u-\Qint{\beta}(u))\nabla \eta_{\vnode}^{k,l} \,   {\dap }dx\\
&\lesssim   l^{-1}\TwoNorm{ \nabla u }{ \dmn \backslash {\omega}_{\vnode,k-l-2} }^2 +(l H)^{-1}\TwoNorm{ \nabla u }{ \dmn \backslash {\omega}_{\vnode,k-l-1} }  \TwoNorm{u-\Qint{\beta}(u) }{\Omega \backslash  {\omega}_{\vnode,k-l-1} }   \\
&\lesssim l^{-1}  \TwoNorm{ \nabla u }{ \dmn \backslash {\omega}_{\vnode,k-l-2} }^2 .
\end{align*}	
}
%
On the last term we used the approximation property \eqref{stableproj}.  Successive applications of 
the above estimate leads to 
\begin{align*}
\TwoNorm{\nabla u}{\cilt\backslash {\omega}_{\vnode,k}}^2&\lesssim  l^{-1} \TwoNorm{ \nabla u }{ \cilt \backslash {\omega}_{\vnode,k-l-2} }^2
\lesssim  l^{- \lfloor \frac{k-1}{l+2}  \rfloor} \TwoNorm{ \nabla u }{ \cilt \backslash {\omega}_{\vnode,1} }^2
\lesssim  l^{- \lfloor \frac{k-1}{l+2}  \rfloor} \TwoNorm{ \nabla u }{ \cilt  }^2.
\end{align*}
Finally, noting that
\[
\left\lfloor \frac{k-1}{l+2}  \right\rfloor = \left\lceil \frac{k-l-2}{l+2}  \right\rceil \geq \frac{k}{l+2} -1,
\]
taking $\theta=l^{-\frac{1}{l+2}}$ yields the result.  \qed
\end{proof}

\par
We now are ready to restate our result on the error introduced from localization. This is merely Lemma \ref{localglobal.derp} restated and proven.
When $k$ is sufficiently large so that the corrector problem is all of $\cilt$, we denote $Q^{\beta}_{\vnode,k}=Q^{\beta}_{\vnode,\cilt}$.
Let $u_{H}\in {V}_{H}$, let $Q^{\beta}_{k}$ be constructed from \eqref{Qcorrector}, and $Q^{\beta}_{\cilt}$ defined to be the``ideal" corrector without truncation, then 
\begin{align}\label{localglobaleq}
	&	\TwoNorm{\nabla( Q^{\beta}_{\dmn}(u_{H})-Q^{\beta}_{k}(u_{H}))  }{\dmn} 
 	\lesssim  k^{\frac{d}{2}} \theta^{k  }   \TwoNorm{ \nabla u_{H} }{\dmn}.
\end{align}
Again we use techniques standard at this point in the view of \cite{brown2017numerical}, but presented for completeness.
\begin{proof}[Proof of Lemma \ref{localglobal.derp}]
	We denote $v= Q^{\beta}_{\dmn}(u_{H})-Q^{\beta}_{k}(u_{H})\in {V}_{\beta}^f$, subsequently $\Qint{\beta}(v)=0$.
	Taking the cut-off function $\eta_{\vnode}^{k,1}$ we have 
	\begin{align}
	\label{term1}
	\TwoNorm{\nabla v}{\dmn}^2 &\lesssim \sum_{\vnode\in {\cal N}_{int}}\int_{\dmn}A \nabla( Q^{\beta}_{\vnode,\dmn}(u_{H})-Q^{\beta}_{\vnode,k}(u_{H}))\nabla (v(1-\eta_{\vnode}^{k,1}))\,  { \dap } dx \\
	\label{term2}
	&+\sum_{\vnode\in {\cal N}_{int}}\int_{\dmn}A \nabla( Q^{\beta}_{\vnode,\dmn}(u_{H})-Q^{\beta}_{\vnode,k}(u_{H}))\nabla (v\eta_{\vnode}^{k,1})\,   { \dap } dx.
	\end{align}
	Estimating the right hand side of \eqref{term1} for each $\vnode$, we have, using the boundedness of $A$
\begin{align*}
&\int_{\cilt}A\nabla( Q^{\beta}_{\vnode,\cilt}(u_{H})-Q^{\beta}_{\vnode,k}(u_{H}))\nabla (v(1-\eta_{\vnode}^{k,1}))\,  { \dap } \,dx\\
&\qquad\lesssim  \TwoNorm{\nabla( Q^{\beta}_{\vnode,\cilt}(u_{H})-Q^{\beta}_{\vnode,k}(u_{H}))}{\cilt} \left(\TwoNorm{\nabla v}{{\omega}_{\vnode,k}}
+  \TwoNorm{v\nabla (1-\eta_{\vnode}^{k,1}))}{{\omega}_{\vnode,k}\backslash {\omega}_{\vnode,k-1}}   \right)\\
&\qquad\lesssim \TwoNorm{\nabla( Q^{\beta}_{\vnode,\cilt}(u_{H})-Q^{\beta}_{\vnode,k}(u_{H}))}{\cilt} \left(\TwoNorm{\nabla v}{{\omega}_{\vnode,k}} 
+  H^{-1}\TwoNorm{v-\Qint{\beta}(v)}{{\omega}_{\vnode,k}\backslash {\omega}_{\vnode,k-1}}   \right)\\
&\qquad\lesssim  \TwoNorm{\nabla( Q^{\beta}_{\vnode,\cilt}(u_{H})-Q^{\beta}_{\vnode,k}(u_{H}))}{\cilt} \TwoNorm{\nabla v}{{\omega}_{\vnode,k+1}}.
\end{align*}
	As in the proof of Lemma \ref{decaylemma}, we denote $\tilde{v}=\eta_{\vnode}^{k,1} v -\Qint{\beta}(\eta_{\vnode}^{k,1} v )\in {V}_{\beta}^f(\dmn\backslash {\omega}_{\vnode,k-2})$ 
	and so $\tilde{v}$ satisfies  
	\begin{align*}
	\int_{\dmn}A \nabla( Q^{\beta}_{\vnode,\cilt}(u_{H})-Q^{\beta}_{\vnode,k}(u_{H}))\nabla\tilde{v}   \,  { \dap }dx=0.
	\end{align*}
	We have now the estimate for \eqref{term2} for $\vnode\in {\cal N}_{int}$ using the above identity and \eqref{qiestimate}
	\begin{align*}
	\int_{\dmn}A \nabla( Q^{\beta}_{\vnode,\dmn}(u_{H})-Q^{\beta}_{\vnode,k}(u_{H}))\nabla (v\eta_{\vnode}^{k,1}-\tilde{v})\,   { \dap }dx
	\lesssim  \TwoNorm{ \nabla( Q^{\beta}_{\vnode,\dmn}(u_{H})-Q^{\beta}_{\vnode,k}(u_{H}))  }{ \dmn  }\TwoNorm{ \nabla v }{ {\omega}_{\vnode ,k+2 } }
	\end{align*}
Combing the estimates for \eqref{term1} and \eqref{term2} we obtain 
\begin{align}
\TwoNorm{\nabla v}{\cilt}^2
&\lesssim \sum_{\vnode\in {\cal N}_{int}} 
\TwoNorm{\nabla( Q^{\beta}_{\vnode,\cilt}(u_{H})-Q^{\beta}_{\vnode,k}(u_{H}))}{\cilt}  \TwoNorm{\nabla v}{{\omega_{\vnode,k+2}}}\nonumber\\
\label{vestimate}
&\lesssim  k^{\frac{d}{2}} \left(\sum_{\vnode\in {\cal N}_{int}}
\TwoNorm{ \nabla( Q^{\beta}_{\vnode,\cilt}(u_{H})-Q^{\beta}_{\vnode,k}(u_{H}))  }{ \cilt  }^2\right)^{\frac{1}{2}}\TwoNorm{ \nabla v }{\cilt },
\end{align}
supposing that 
$\# \{\vnode'\in {\cal N}_{int}|{\omega}_{\vnode'}\subset {\omega}_{\vnode,k+2}\}\lesssim k^{d}$,
as is guaranteed by quasi-uniformity of the coarse-grid.
\par

For $\vnode \in{\cal N}_{int}$,  we  estimate $\TwoNorm{ \nabla( Q^{\beta}_{\vnode,\dmn}(u_{H})-Q^{\beta}_{\vnode,k}(u_{H}))  }{ \dmn  }$ and 
we use the Galerkin orthogonality of the local problem, that is 
\begin{align}\label{galerkinlocal}
\TwoNorm{ \nabla( Q^{\beta}_{\vnode,\dmn}(u_{H})-Q^{\beta}_{\vnode,k}(u_{H}))  }{ \dmn  }\leq \inf_{q_\vnode\in {V}_{\beta}^f({\omega}_{\vnode,k}) } \TwoNorm{ \nabla( Q^{\beta}_{\vnode,\cilt}(u_{H})-q_\vnode)  }{ \dmn  }.
\end{align}
Let 
$q_{\vnode}=(1-\eta^{(k-1),1}_{\vnode})Q^{\beta}_{\vnode,\dmn}(u_{H})-\Qint{\beta}((1-\eta^{(k-1),1}_{\vnode})Q_{\vnode,\dmn}(u_{H})  )\in {V}^f({\omega}_{\vnode,k})$, 
we have 
\begin{align*}
&\TwoNorm{ \nabla( Q^{\beta}_{\vnode,\dmn}(u_{H})-Q^{\beta}_{\vnode,k}(u_{H}))  }{ \dmn  }\leq 
\TwoNorm{ \nabla( \eta^{(k-1),1}_{\vnode}Q^{\beta}_{\vnode,\dmn}(u_{H})+\Qint{\beta}((1-\eta^{(k-1),1}_{\vnode})Q^{\beta}_{\vnode,\dmn}(u_{H})  )) }{ \dmn  }\\
&\lesssim \TwoNorm{ \nabla Q^{\beta}_{\vnode,\dmn}(u_{H}) }{\dmn \backslash {\omega}_{\vnode,k-2} }+ 
\TwoNorm{ \nabla(\Qint{\beta}((1-\eta^{k-1,1}_{\vnode})Q^{\beta}_{\vnode,\dmn}(u_{H})  )) }{ \Omega }.
\end{align*}
 {
Using $\Qint{\beta}((1-\eta^{k-1,1}_{\vnode})Q^{\beta}_{\vnode,\dmn}(u_{H})  )=-\Qint{\beta}(\eta^{k-1,1}_{\vnode}Q^{\beta}_{\vnode,\dmn}(u_{H})  )$ and Lemma \ref{qi} on the second term, and then Lemma \ref{decaylemma}, we arrive at 
}
\begin{align*}
&\TwoNorm{ \nabla( Q^{\beta}_{\vnode,\cilt}(u_{H})-Q^{\beta}_{\vnode,k}(u_{H}))  }{ \cilt  }^2
\lesssim  \TwoNorm{ \nabla Q^{\beta}_{\vnode,\cilt }(u_{H}) }{ \cilt \backslash {\omega}_{\vnode,k-4} }^2\lesssim  \theta^{2(k-4) } \TwoNorm{ \nabla Q^{\beta}_{\vnode,\cilt}(u_{H}) }{ \cilt  }^2.
\end{align*} 
From the definition of $Q^{\beta}_{\vnode,\cilt}$ from \eqref{Qcorrector} with global corrector patches,
we get
\begin{align*}
\TwoNorm{ \nabla( Q_{\vnode,\cilt}(u_{H})-Q_{\vnode,k}(u_{H}))  }{ \cilt  }^2
\lesssim \theta^{2k} \TwoNorm{ \nabla u_{H} }{ {\omega}_{\vnode} }^2,
\end{align*}
 {
where we used the bounds from Proposition \ref{correctorPropBound} modified for $Q^{\beta}_{\vnode,\cilt}$ from \eqref{Qcorrector}
with localized right hand side, hence localized upper bounds.
}
Thus, summing over all $\vnode\in {\cal N}_{int}$  and combining the above with \eqref{vestimate} concludes the proof. \qed	
\end{proof}

\bibliographystyle{abbrv}	
\bibliography{HMP_references.bib}

\begin{thebibliography}{10}

\bibitem{Abdulle:E:Engquist:Vanden-Eijnden:2012}
A.~Abdulle, W.~E, B.~Engquist, and E.~Vanden-Eijnden.
\newblock The heterogeneous multiscale method.
\newblock {\em Acta Numerica}, 21:1--87, 2012.

\bibitem{Henningwave}
A.~Abdulle and P.~Henning.
\newblock Localized orthogonal decomposition method for the wave equation with
  a continuum of scales.
\newblock {\em Mathematics of Computation}, 86(304):549--587, 2017.

\bibitem{agnelli2014posteriori}
J.~P. Agnelli, E.~M. Garau, and P.~Morin.
\newblock A posteriori error estimates for elliptic problems with {D}irac
  measure terms in weighted spaces.
\newblock {\em ESAIM: Mathematical Modelling and Numerical Analysis},
  48(6):1557--1581, 2014.

\bibitem{babuvska1971error}
I.~Babu{\v{s}}ka.
\newblock Error-bounds for finite element method.
\newblock {\em Numerische Mathematik}, 16(4):322--333, 1971.

\bibitem{bramble2002stability}
J.~Bramble, J.~Pasciak, and O.~Steinbach.
\newblock On the stability of the {$L^2$} projection in {$H^1(\Omega)$}.
\newblock {\em Mathematics of Computation}, 71(237):147--156, 2002.

\bibitem{brenner2007mathematical}
S.~C. Brenner and L.~R. Scott.
\newblock {\em The mathematical theory of finite element methods}, volume~15.
\newblock Springer Science \& Business Media, 2007.

\bibitem{brown2016multiscaleelastic}
D.~L. Brown and D.~Gallistl.
\newblock Multiscale sub-grid correction method for time-harmonic
  high-frequency elastodynamics with wavenumber explicit bounds.
\newblock {\em arXiv preprint arXiv:1608.04243}, 2016.

\bibitem{Brown2017}
D.~L. Brown, D.~Gallistl, and D.~Peterseim.
\newblock {\em Multiscale Petrov-Galerkin Method for High-Frequency
  Heterogeneous Helmholtz Equations}, pages 85--115.
\newblock Springer International Publishing, Cham, 2017.

\bibitem{brown2017numerical}
D.~L. Brown, J.~Gedicke, and D.~Peterseim.
\newblock Numerical homogenization of heterogeneous fractional {L}aplacians.
\newblock {\em arXiv preprint arXiv:1709.00730}, 2017.

\bibitem{brown2016multiscale}
D.~L. Brown and D.~Peterseim.
\newblock A multiscale method for porous microstructures.
\newblock {\em Multiscale Modeling \& Simulation}, 14(3):1123--1152, 2016.

\bibitem{cabre2010positive}
X.~Cabr{\'e} and J.~Tan.
\newblock Positive solutions of nonlinear problems involving the square root of
  the {L}aplacian.
\newblock {\em Advances in Mathematics}, 224(5):2052--2093, 2010.

\bibitem{caffarelli2007extension}
L.~Caffarelli and L.~Silvestre.
\newblock An extension problem related to the fractional {L}aplacian.
\newblock {\em Communications in partial differential equations},
  32(8):1245--1260, 2007.

\bibitem{caffarelli2016fractional}
L.~A. Caffarelli and P.~R. Stinga.
\newblock Fractional elliptic equations, {C}accioppoli estimates and
  regularity.
\newblock In {\em Annales de l'Institut Henri Poincare (C) Non Linear
  Analysis}, volume~33, pages 767--807. Elsevier, 2016.

\bibitem{capella2011regularity}
A.~Capella, J.~D{\'a}vila, L.~Dupaigne, and Y.~Sire.
\newblock Regularity of radial extremal solutions for some non-local semilinear
  equations.
\newblock {\em Communications in Partial Differential Equations},
  36(8):1353--1384, 2011.

\bibitem{chen2009well}
Z.~Chen and Y.~Zhang.
\newblock Well flow models for various numerical methods.
\newblock {\em International Journal of Numerical Analysis \& Modeling}, 6(3),
  2009.

\bibitem{clement1975approximation}
P.~Cl{\'e}ment.
\newblock Approximation by finite element functions using local regularization.
\newblock {\em Revue fran{\c{c}}aise d'automatique, informatique, recherche
  op{\'e}rationnelle. Analyse num{\'e}rique}, 9(R2):77--84, 1975.

\bibitem{d2012finite}
C.~D'Angelo.
\newblock Finite element approximation of elliptic problems with {D}irac
  measure terms in weighted spaces: applications to one-and three-dimensional
  coupled problems.
\newblock {\em SIAM Journal on Numerical Analysis}, 50(1):194--215, 2012.

\bibitem{d2008coupling}
C.~D'Angelo and A.~Quarteroni.
\newblock On the coupling of 1d and 3d diffusion-reaction equations:
  application to tissue perfusion problems.
\newblock {\em Mathematical Models and Methods in Applied Sciences},
  18(08):1481--1504, 2008.

\bibitem{durlofsky2005upscaling}
L.~J. Durlofsky.
\newblock Upscaling and gridding of fine scale geological models for flow
  simulation.

\bibitem{ErnGuermond2015}
A.~Ern and J.-L. Guermond.
\newblock Finite element quasi-interpolation and best approximation.
\newblock {\em ESAIM. Mathematical Modelling and Numerical Analysis},
  51(4):1367--1385, 2017.

\bibitem{fabes1982local}
E.~B. Fabes, C.~E. Kenig, and R.~P. Serapioni.
\newblock The local regularity of solutions of degenerate elliptic equations.
\newblock {\em Communications in Statistics-Theory and Methods}, 7(1):77--116,
  1982.

\bibitem{gilman2003practical}
J.~R. Gilman.
\newblock Practical aspects of simulation of fractured reservoirs.
\newblock 2003.

\bibitem{gol2009weighted}
V.~Gol'dshtein and A.~Ukhlov.
\newblock Weighted sobolev spaces and embedding theorems.
\newblock {\em Transactions of the American Mathematical Society},
  361(7):3829--3850, 2009.

\bibitem{gong2008upscaling}
B.~Gong, M.~Karimi-Fard, L.~J. Durlofsky, et~al.
\newblock Upscaling discrete fracture characterizations to dual-porosity,
  dual-permeability models for efficient simulation of flow with strong
  gravitational effects.
\newblock {\em SPE Journal}, 13(01):58--67, 2008.

\bibitem{haroske2008atomic}
D.~D. Haroske and I.~Piotrowska.
\newblock Atomic decompositions of function spaces with {M}uckenhoupt weights,
  and some relation to fractal analysis.
\newblock {\em Mathematische Nachrichten}, 281(10):1476--1494, 2008.

\bibitem{heinonen2012nonlinear}
J.~Heinonen, T.~Kilpel{\"a}inen, and O.~Martio.
\newblock {\em Nonlinear Potential Theory of Degenerate Elliptic Equations}.
\newblock Dover Books on Mathematics Series. Dover Publications, 2012.

\bibitem{henning2014localized}
P.~Henning and A.~M{\aa}lqvist.
\newblock Localized orthogonal decomposition techniques for boundary value
  problems.
\newblock {\em SIAM Journal on Scientific Computing}, 36(4):A1609--A1634, 2014.

\bibitem{HMP12}
P.~Henning, A.~M{\aa}lqvist, and D.~Peterseim.
\newblock A localized orthogonal decomposition method for semi-linear elliptic
  problems.
\newblock {\em ESAIM. Mathematical Modelling and Numerical Analysis},
  48(5):1331--1349, 2014.

\bibitem{Henning.Morgenstern.Peterseim:2014}
P.~{Henning}, P.~{Morgenstern}, and D.~{Peterseim}.
\newblock {Multiscale Partition of Unity}.
\newblock In M.~Griebel and M.~A. Schweitzer, editors, {\em Meshfree Methods
  for Partial Differential Equations {VII}}, volume 100 of {\em Lecture Notes
  in Computational Science and Engineering}. Springer, 2014.
\newblock {Also available as INS Preprint No. 1315}.

\bibitem{HP13}
P.~Henning and D.~Peterseim.
\newblock Oversampling for the {M}ultiscale {F}inite {E}lement {M}ethod.
\newblock {\em Multiscale Modeling \& Simulation. A SIAM Interdisciplinary
  Journal}, 11(4):1149--1175, 2013.

\bibitem{Hou:Wu:1997}
T.~Y. Hou and X.-H. Wu.
\newblock A multiscale finite element method for elliptic problems in composite
  materials and porous media.
\newblock {\em Journal of Computational Physics}, 134(1):169--189, 1997.

\bibitem{MR1660141}
T.~J.~R. Hughes, G.~R. Feij{\'o}o, L.~Mazzei, and J.-B. Quincy.
\newblock The variational multiscale method---a paradigm for computational
  mechanics.
\newblock {\em Computer Methods in Applied Mechanics and Engineering},
  166(1-2):3--24, 1998.

\bibitem{MR2300286}
T.~J.~R. Hughes and G.~Sangalli.
\newblock Variational multiscale analysis: the fine-scale {G}reen's function,
  projection, optimization, localization, and stabilized methods.
\newblock {\em SIAM Journal on Numerical Analysis}, 45(2):539--557, 2007.

\bibitem{Kornhuber.Peterseim.Yserentant:2016}
R.~{Kornhuber}, D.~{Peterseim}, and H.~{Yserentant}.
\newblock {An analysis of a class of variational multiscale methods based on
  subspace decomposition}.
\newblock {\em arXiv preprint arXiv:1608.04081}, 2016.

\bibitem{kufner1985weighted}
A.~Kufner.
\newblock {\em Weighted Sobolev Spaces}.
\newblock Teubner-Texte zur Mathematik. B.G. Teubner, 1985.

\bibitem{li2015effective}
J.~Li, Z.~Lei, G.~Qin, and B.~Gong.
\newblock Effective local-global upscaling of fractured reservoirs under
  discrete fractured discretization.
\newblock {\em Energies}, 8(9):10178--10197, 2015.

\bibitem{MP11}
A.~M{\aa}lqvist and D.~Peterseim.
\newblock Localization of elliptic multiscale problems.
\newblock {\em Mathematics of Computation}, 83(290):2583--2603, 2014.

\bibitem{MP12}
A.~M{\aa}lqvist and D.~Peterseim.
\newblock {Computation of eigenvalues by numerical upscaling}.
\newblock {\em Numerische Mathematik}, 130(2):337--361, 2015.

\bibitem{melenkApel}
J.~Melenk and T.~Apel.
\newblock Interpolation and quasi-interpolation in h- and hp-version finite
  element spaces.
\newblock In E.~Stein, R.~de~Borst, and T.~Hughes, editors, {\em Encyclopedia
  of Computational Mechanics}. 2017.

\bibitem{muckenhoupt1972weighted}
B.~Muckenhoupt.
\newblock Weighted norm inequalities for the hardy maximal function.
\newblock {\em Transactions of the American Mathematical Society}, pages
  207--226, 1972.

\bibitem{nekvinda1993characterization}
A.~Nekvinda.
\newblock Characterization of traces of the weighted sobolev space
  ${W^{1,p}(\Omega, d_M^{\epsilon})}$ on $ {M}$.
\newblock {\em Czechoslovak Mathematical Journal}, 43(4):695--711, 1993.

\bibitem{nochetto2015pde}
R.~H. Nochetto, E.~Ot{\'a}rola, and A.~J. Salgado.
\newblock A pde approach to fractional diffusion in general domains: a priori
  error analysis.
\newblock {\em Foundations of Computational Mathematics}, 15(3):733--791, 2015.

\bibitem{nochetto2016piecewise}
R.~H. Nochetto, E.~Ot{\'a}rola, and A.~J. Salgado.
\newblock Piecewise polynomial interpolation in {M}uckenhoupt weighted
  {S}obolev spaces and applications.
\newblock {\em Numerische Mathematik}, 132(1):85--130, 2016.

\bibitem{peaceman1983interpretation}
D.~W. Peaceman et~al.
\newblock Interpretation of well-block pressures in numerical reservoir
  simulation with nonsquare grid blocks and anisotropic permeability.
\newblock {\em Society of Petroleum Engineers Journal}, 23(03):531--543, 1983.

\bibitem{Peterseim:2015}
D.~Peterseim.
\newblock Variational multiscale stabilization and the exponential decay of
  fine-scale correctors.
\newblock 114:341--367, 2016.

\bibitem{scott1973finite}
L.~R. Scott.
\newblock Finite element convergence for singular data.
\newblock {\em Numerische Mathematik}, 21(4):317--327, 1973.

\bibitem{scott1990finite}
L.~R. Scott and S.~Zhang.
\newblock Finite element interpolation of nonsmooth functions satisfying
  boundary conditions.
\newblock {\em Mathematics of Computation}, 54(190):483--493, 1990.

\bibitem{stampacchia1963equations}
G.~Stampacchia.
\newblock Equations elliptiques du second ordrea coefficients discontinus.
\newblock {\em S{\'e}minaire Jean Leray}, 3:1--77, 1963.

\end{thebibliography}

\end{document}